\documentclass[12pt]{article}

\usepackage{amsmath,amssymb,amsfonts,amsthm}
\usepackage{graphics}
\usepackage{epsfig}
 \usepackage{amsmath}
\usepackage{amsfonts}

\font\elevensf=cmss10 scaled\magstephalf

\newtheorem{theorem} {{\elevensf THEOREM}}[section]
\newtheorem{proposition} {{\elevensf PROPOSITION}}[section]

\newtheorem{lemma} {{\elevensf LEMMA}}[section]

\newtheorem{example} {{\elevensf EXAMPLE}}[section]

\newtheorem{remark} {{\elevensf REMARK}}[section]

\renewcommand\qed{$\blacksquare$}

\def\CC{{\rm \kern.24em \vrule width.02em height1.4ex depth-.05ex \kern-.26emC}}
\def\epsilon{\varepsilon}
\def\TagOnRight

\def\AA{{it I} \hskip-3pt{\tt A}}

\def\QQ{\rlap {\raise 0.4ex \hbox{$\scriptscriptstyle |$}} {\hskip -0.1em Q}}

\catcode`\@=11

\newcommand{\lb}{\left(}
\newcommand{\rb}{\right)}

\def\theequation{\@arabic{\c@section}.\@arabic{\c@equation}}

\begin{document}

\baselineskip 14pt
\parindent.4in
\catcode`\@=11 

\begin{center}
{\Huge \bf Bloch wave spectral analysis in the class of generalized
Hashin-Shtrikman micro-structures } \\[5mm]
{\bf \textbf{Loredana B\u{a}lilescu, Carlos Conca, Tuhin Ghosh, Jorge San Mart{\'{\i}}n and Muthusamy Vanninathan}} 
\end{center} 

\begin{abstract}
\noindent
In this paper, we use spectral methods by introducing the Bloch waves 
to study the homogenization process in the non-periodic class of generalized 
Hashin-Shtrikman micro-structures \cite[page no. 281]{T}, which incorporates both 
translation and dilation with a family of scales, including one subclass of laminates. We establish the classical homogenization result with providing the spectral
representation of the homogenized coefficients. 
It offers a new lead towards extending the Bloch spectral analysis 
in the non-periodic, non-commutative class of micro-structures. 
\end{abstract}

\section{Introduction}\label{intro}
\setcounter{equation}{0} 
We start with a homogenization process in the class of elliptic boundary value 
problems where the heterogeneous media is governed by the well known generalized 
Hashin-Shtrikman micro-structures \cite[page no. 281]{T}. Following by the its 
construction the coefficients are invariant in a certain way of both translation 
and dilation of the medium with incorporating a family of small scales 
$\{\varepsilon_p>0\}_p$. Precise understanding of the heterogeneity will be made 
shortly as a beginning remark we want to point it out that, here the 
coefficients are need not be periodic with a small period $(\epsilon >0)$. Periodic
micro-structure incorporates uniform translation and uniform dialation with respect 
to only one scale $\epsilon$, where as in Hashin-Shtrikman construction 
incorporates non-uniform translation and dilation with a family of 
scales $\{\epsilon_p\}_p$. For a better understanding of its homogenization process we refer 
\cite{T,JKO}, etc. In terms of Murat-Tartar's \cite{MT1,MT2} 
theory of $H$-convergence, the main result says that the (weak) limit of such 
solutions resolves a suitable boundary value problem which has constant 
coefficients that represent what is known as homogenized medium. The theory of 
$H$-convergence is very generic, it just proves the existence of the various
sub-sequential homogenized limits. Although, in a particular class of 
Hashin-Shtrikman micro-structures, all the sub-sequential homogenized limits are same
and can be calculated. It is given by some integral representation \eqref{mkl}
below like as we have in the periodic setup. In this paper we are interested about 
the methods of homogenization process and its other aspects. Let us start with 
the classical model problem of heat conduction in composite materials.\\
Let us consider a family of inhomogeneous media occupying a certain bounded 
region $\Omega$ in $\mathbb{R}^N$, parameterized by a small parameter $\epsilon$
and represented by $N\times N$ matrices of real-valued functions 
$A^{\epsilon}(x) = [a_{kl}^{\epsilon}(x)]$ defined on $\Omega$ and satisfying $A^{\epsilon}\in\mathcal{M}(\alpha,\beta;\Omega),$ for some $0<\alpha <\beta$, i.e. 
\begin{equation*}
a_{kl}^{\epsilon}(x)=a_{lk}^{\epsilon}(x),\ \ (A^{\epsilon}(x)\xi,\xi) \geq \alpha|\xi|^2,\ \ |A^{\epsilon}(x)\xi|\leq \beta|\xi|\mbox{ for any } \xi \in \mathbb{R}^N,\mbox{ a.e. }x\in\Omega.
\end{equation*}
The small scale parameter $\epsilon$ defines the heterogeneity of the 
conductivities in the medium. Following the Fourier law one considers 
the following boundary value problem in the general case:
\begin{equation*}\begin{aligned}
\mathcal{A}^{\epsilon}u^{\epsilon}(x) = -div(A^{\epsilon}(x)\nabla u^{\epsilon}(x)) &= f(x) \quad\mbox{in }\Omega,\\
u^{\epsilon}&=0 \quad\mbox{on }\partial\Omega,
\end{aligned}\end{equation*}
where $f$ is given in $H^{-1}(\Omega)$. 
\begin{example}[Periodic micro-structures \cite{A}]\label{periodic}

Let $Y$ denotes the unit cube $[0,1]^N$ in $\mathbb{R}^N$.
Let $A_Y(y)=[a^Y_{kl}(y)]_{1\leq k,l\leq N}\in \mathcal{M}(\alpha,\beta,Y)$  be such that $a^Y_{kl}(y)$ are $Y$-periodic functions
$ \forall k,l =1,2..,N.$ Now we set 
\begin{equation*}A^{\epsilon}_Y(x) = [a_{kl}^{\epsilon}(x)]= \Big[a^Y_{kl}\Big(\frac{x}{\epsilon}\Big)\Big]\end{equation*} and extend 
it to the whole $\mathbb{R}^N$ by $\epsilon$-periodicity with a small period of scale $\epsilon$ 
and then restricting $A^{\epsilon}$ in particular on $\Omega$ is known as periodic micro-structures.
\end{example}
\begin{example}[Hashin-Shtrikman micro-structures {\cite[Page no. 281]{T}}] \label{hsm}

Let $\omega \subset \mathbb{R}^N$ be a bounded open with Lipschitz boundary.
Let $A_\omega(y)=[a^\omega_{kl}(y)]_{1\leq k,l\leq N} \in \mathcal{M}(\alpha,\beta,\omega)$  
be such that after extending $A_\omega$ by $A_\omega(y) = M$ for $y \in \mathbb{R}^N\smallsetminus \omega$, where $M \in L_{+}(\mathbb{R}^N ; \mathbb{R}^N)$
(i.e. $M = [m_{kl}]_{1\leq k,l \leq N}$ is a constant positive definite $N\times N$ matrix), 
if for each $\lambda \in \mathbb{R}^N$ there exists $w_{\lambda}\in  H^{1}_{loc}(\mathbb{R}^N)$ satisfying 
\begin{equation}\label{hsw}
- div (A_\omega(y)\nabla w_{\lambda}(y)) = 0 \quad\mbox{in }\mathbb{R}^N,\quad  w_{\lambda}(y) = (\lambda, y) \quad\mbox{in }\mathbb{R}^N \smallsetminus \omega,
\end{equation}
then $A_\omega$ is said to be \textit{equivalent} to $M$.\\
\\
Then one uses a sequence of Vitali coverings of $\Omega$ by reduced copies of $\omega,$
\begin{equation}\label{hso}  meas(\Omega \smallsetminus \underset{p\in K}{\cup} (\epsilon_{p,n}\omega + y^{p,n}) = 0, \mbox{ with } \kappa_n = \underset{p\in K}{sup}\ \epsilon_{p,n}\rightarrow 0,\end{equation}
for a finite or countable $K$. These define the micro-structures in $A^{n}_\omega$. One defines for almost everywhere $x\in \Omega$, 
\begin{equation}\label{hs} A^{n}_\omega(x) = A_\omega\Big(\frac{x - y^{p,n}}{\epsilon_{p,n}}\Big) \mbox{ in } \epsilon_{p,n}\omega + y^{p,n},\quad p\in K,\end{equation}
which makes sense since, for each $n$, the sets $\epsilon_{p,n}\omega + y^{p,n},\ p\in K$ are disjoint.
The above construction \eqref{hs} represents the so called \textbf{Hashin-Shtrikman micro-structures}.
\end{example}
There is a vast work in the literature concerning with the limiting procedure
of the above equation in order to to get the homogenized coefficients.
In a large significance, Tartar's Method of oscillating
test functions \cite{MT1,MT2} and Compensated Compactness \cite{M1,T},
which he developed in large part in association with F. Murat, 
stands as a very general to take care of the limiting 
procedure which leads to $H$-convergence formulation for any arbitrary 
micro-structures. However, if the medium is periodically heterogeneous, there
is an alternate procedure to pass to the limit by using the notion asymptotic 
expansions by Bensoussan, Lions and Papanicolaou \cite{BLP}, notion of two-scale (or multi-scale) weak 
convergence by Nguetseng \cite{N} and Allaire \cite{A,A2}. If the 
solution of the boundary value problem is realized as a minimum of a suitable 
energy functional, then the convergence of the solution sequence can be reduced
to $\Gamma$-convergence of the energy functional. This idea is due to De Giorgi and his 
collaborators \cite{DM}. These above mentioned methods are usually categorized as 
Physical space methods. There are also other class of methods which are known as 
Fourier space methods and Phase space methods respectively. The phase space 
theories provide tools suitable to study the behavior
of inhomogeneous media which are qualitatively similar to homogeneous ones 
and their small perturbations. Tartar's H-measures \cite{T1,T2}, Gerard's
micro-local defect measures \cite{PG1,PG2} are known tools towards this 
aspect. And the remaining one the Fourier space methods stands in between 
these two classes of physical space methods and phase space methods. 
Driven by its own Fourier techniques, it ables to provide 
a suitable description produced by periodically modulated heterogeneities in 
the Fourier space and further, towards the limit equation for the homogenized medium. 
The periodic homogenization result was obtained by Conca and 
Vanninathan \cite{CV}, following by their works in the vibrations of fluid-solid structures
\cite{CPV} and the work by Morgan and Babu\u{s}ka 
\cite{MB1,MB2} of the Fourier analytic approach to homogenization problems.
If the medium is homogeneous, Fourier techniques have proved to
be extremely useful both analytically and numerically. 
In the same way, these techniques, if extended properly
to non-homogeneous cases, are expected to be as fruitful.
In this paper, we take an attempt to extend this method
beyond the periodic homogenization, namely to the class of generalized 
Hashin-Shtrikman micro-structures. Of course, the major difference in this is
that they do not form a periodic array in the usual sense. However,
there is still the invariance by the action of translation and dilation
groups, but the action is now different and furthermore, they don't commute each other.
This fact will be exploited to introduce Bloch waves in such non-periodic
structures. This is one of new contributions.
\section{Preliminaries}
\setcounter{equation}{0} 
Let us begin with by recalling few basic ideas behind the homogenization process. 
In particular, we draw a parallel discussion between periodic and Hashin-Shtrikman micro-structures.    
\subsection{H-Convergence}

There are various equivalent definitions of the
homogenized matrix. For general micro-structures, it is defined as
H-limit \cite{T}. If the system admits a Lagrangian density, more
precisely, if the system corresponds to Euler-Lagrange equation of a
minimization problem, then homogenized matrix is defined as
Gamma-convergence limit. In case of periodic micro-structures, the
homogenized matrix is given by an average over the periodic cell involving
the coefficients and the solutions of cell test problems. In the present
case of Hashin-Shtrikman structures, there is an alternate way of viewing the
homogenized matrix. This essential point is well-highlighted by Tartar via
the notion of a homogeneous medium being equivalent to the micro-structured
medium \cite[page no. 281]{T}. We exploit this characterization of the homogenized
matrix in our analysis.\\
\\
Let us start with recalling the definition of \textit{H-convergence}.
Considering a sequence of matrices $A^\epsilon\in \mathcal{M}(\alpha,\beta,\Omega)$ 
is said to converge in the sense of homogenization to an homogenized limit (H-limit) 
matrix $A^{*}\in \mathcal{M}(\alpha,\beta,\Omega),$  i.e. $A^{\epsilon} \xrightarrow{H} A^{*}$, if for any right hand side $f\in H^{-1}(\Omega)$, 
the sequence $u^\epsilon$ is the solution of
\begin{equation*}
\begin{aligned}       
-div(A^\epsilon(x)\mathbb\nabla u^\epsilon(x))&= f(x)\quad\mbox{in } \Omega,\\
                              u^\epsilon(x)&= 0 \quad\mbox{on }\partial\Omega
\end{aligned}
\end{equation*}
satisfies
\begin{equation*} 
\begin{aligned}
 u^{\epsilon}(x) &\rightharpoonup u(x) \quad\mbox{weakly in }H^{1}_{0}(\Omega),\\
 A^{\epsilon}(x)\nabla u^{\epsilon}(x) &\rightharpoonup A^{*}(x)\nabla u(x) \quad\mbox{weakly in }(L^{2}(\Omega))^N, 
\end{aligned}
\end{equation*}
where $u$ is the solution of the homogenized equation 
\begin{equation*} 
\begin{aligned}
-div(A^{*}(x)\nabla u(x))&= f(x) \quad\mbox{in }\Omega, \\
 u(x)&= 0  \quad\mbox{on }\partial\Omega.
\end{aligned}
\end{equation*}
The homogenized limit $A^{*}$ is locally defined and does not depend on the
source term $f$ or the boundary condition on $\partial\Omega$.\\
Let us recall the $H$-convergence limit for periodic and Hashin-Shtrikman micro-structures. 
\paragraph{Periodic homogenization:}
Here we consider $A^{\epsilon}_Y(x) = A_Y(\frac{x}{\epsilon})$, say $y=\frac{x}{\epsilon}$ and 
$A_Y(y)$ is a $Y$-periodic matrix function. \\
Then the homogenized conductivity $A^{*}_Y$ is defined by its entries
\begin{equation}\begin{aligned}\label{prr}
(A^{*}_Y)_{kl} =& \int_Y A_Y(y)(\nabla_y\chi_k+ e_k)\cdot e_l\  dy \\
             =& \int_Y A_Y(y)(\nabla_y \chi_k + e_k)\cdot (\nabla_y\chi_l +e_l)\ dy,
\end{aligned}\end{equation}
where we define  $\chi_k$ through the so-called cell-problems:
for each unit vector $e_{k}$, $\chi_k\in H^1_{\#}(Y)$ solves the following conductivity problem in the periodic unit cell: 
\begin{equation*}
-div_y(A_Y(y)(\nabla_y\chi_k(y)+e_k)) = 0 \quad\mbox{in Y,}\quad y \longmapsto\chi_k(y) \quad\mbox{is $Y$-periodic.}
\end{equation*}
\paragraph{Hashin-Shtrikman constructions:}
Here we consider 
\begin{equation*} A^{n}_\omega(x) = A_\omega\Big(\frac{x - y^{p,n}}{\epsilon_{p,n}}\Big)\quad\mbox{ in  }\epsilon_{p,n}\omega + y^{p,n}\end{equation*} 
defined through \eqref{hso} and \eqref{hs} and $A_\omega$ is equivalent to $M$. That for each  $\lambda\in\mathbb{R}^N$, there exists $w_{\lambda}\in H^1_{loc}(\mathbb{R}^N)$ defined as in \eqref{hsw} solving 
\begin{equation}\begin{aligned}\label{hsp}
-div(A_\omega\nabla w_{\lambda}(y)) &= 0 \mbox{ in }\omega, \\
 w_{\lambda}(y) - \langle\lambda, y\rangle \in H^1_0(\omega),&\quad \langle A_\omega\nabla w_{\lambda}, \nu\rangle = \langle M\lambda,\nu\rangle \mbox{ on }\partial\omega.
\end{aligned}\end{equation}
We remark that, as viewed in this manner on $\omega$, the above system is overdetermined since there exist too many boundary conditions.\\
Following that, one defines $u^{n} \in H^1(\Omega)$ by
\begin{equation}\label{uB}
v^{n}(x) = \epsilon_{p,n}w_{\lambda}\Big(\frac{x-y^{p,n}}{\epsilon_{p,n}}\Big)+ (\lambda, y^{p,n})\quad\mbox{ in  }\epsilon_{p,n}\omega + y^{p,n}.
\end{equation}
Then, one has the following convergences \cite[Page no. 283]{T}:
\begin{equation*}\begin{aligned}
& v^{n}(x) \rightharpoonup (\lambda,x) \mbox{ weakly in }H^1(\Omega;\mathbb{R}^N),\\
& A^{n}_\omega\nabla v^{n}(x)\rightharpoonup M\lambda \mbox{ weakly in } L^2(\Omega;\mathbb{R}^N),\\
& -div( A^{n}_\omega(x)\nabla v^{n}(x)) =\ 0 \quad\mbox{in }\Omega. 
\end{aligned}\end{equation*}
So, by the definition of $H$-convergence, one has the following convergence of the entire sequence
\begin{equation*} A^{n}_\omega \xrightarrow{H \mbox{ converges }} M, \end{equation*}
where $M\in L_{+}(\mathbb{R}^N,\mathbb{R}^N)$ is a positive definite matrix equivalent to $A$.\\
\\
We have the following integral representation similar to \eqref{prr}:
\begin{equation}
\begin{aligned}\label{mkl}
m_{kl} =&\ \frac{1}{|\omega|}\int_{\omega} A_\omega(y)\nabla w_{e_k}\cdot e_l\ dy \\
 =&\ \frac{1}{|\omega|}\int_{\omega} A_\omega(y)\nabla w_{e_k}\cdot\nabla w_{e_l}\ dy, 
\end{aligned}
\end{equation}
where $w_{e_k}, w_{e_l}$ are the solution of \eqref{hsp} for $\lambda= e_k$ and $\lambda=e_l$, respectively.\\
Notice that the above expressions are simply coming through doing integration by-parts in \eqref{hsp} multiplied 
by $y_l$ and $w_{e_l}(y),$ respectively. 
\begin{remark}
It can be also seen that, for every $n$, we have
\begin{equation*}
M\lambda\cdot\lambda \ = \ \frac{1}{|\Omega|} \int_{\Omega} A^{n}_\omega(x)\nabla v^{n}(x)\cdot\nabla v^{n}(x)\ dx,
\end{equation*}
because\begin{equation*}\begin{aligned}
\frac{1}{|\Omega|}& \int_{\Omega} A^{n}_\omega(x)\nabla v^{n}(x)\cdot\nabla v^{n}(x) dx \\
&= \ \frac{1}{|\Omega|}\sum_p \int_{\epsilon_{p,n}\omega +y^{p,n}} A_\omega\Big(\frac{x-y^{p,n}}{\epsilon_{p,n}}\Big)\nabla w_{\lambda}\Big(\frac{x-y^{p,n}}{\epsilon_{p,n}}\Big)\cdot\nabla w_{\lambda}\Big(\frac{x-y^{p,n}}{\epsilon_{p,n}}\Big)dx\\
&= \  \frac{1}{|\Omega|}\sum_p\epsilon_{p,n}^N \int_{\omega} A(x)\nabla w_{\lambda}\cdot\nabla w_{\lambda} dx\ =\ \frac{1}{|\omega|}\int_{\omega} A(x)\nabla w_{\lambda}\cdot\nabla w_{\lambda} dx  \mbox{ as }\sum_p \epsilon_{p,n}^N = \frac{|\Omega|}{|\omega|}.
\end{aligned}\end{equation*} 
\end{remark}

One important thing to notice is the $H$-limit does not depend on the choice of translations $y^{p,n}$ and 
the scales $\epsilon_{p,n}$ as long as they are imposed to satisfy the Vitali covering theorem's criteria \eqref{hso}. 
Where as, the higher order approximation of the medium, for example, the fourth order
approximation or \textit{Burnett coefficient}s (see \cite{TV2}), depends on the particular choice of the scales. 
\begin{figure}
 \begin{center}
  \includegraphics[width = 14cm]{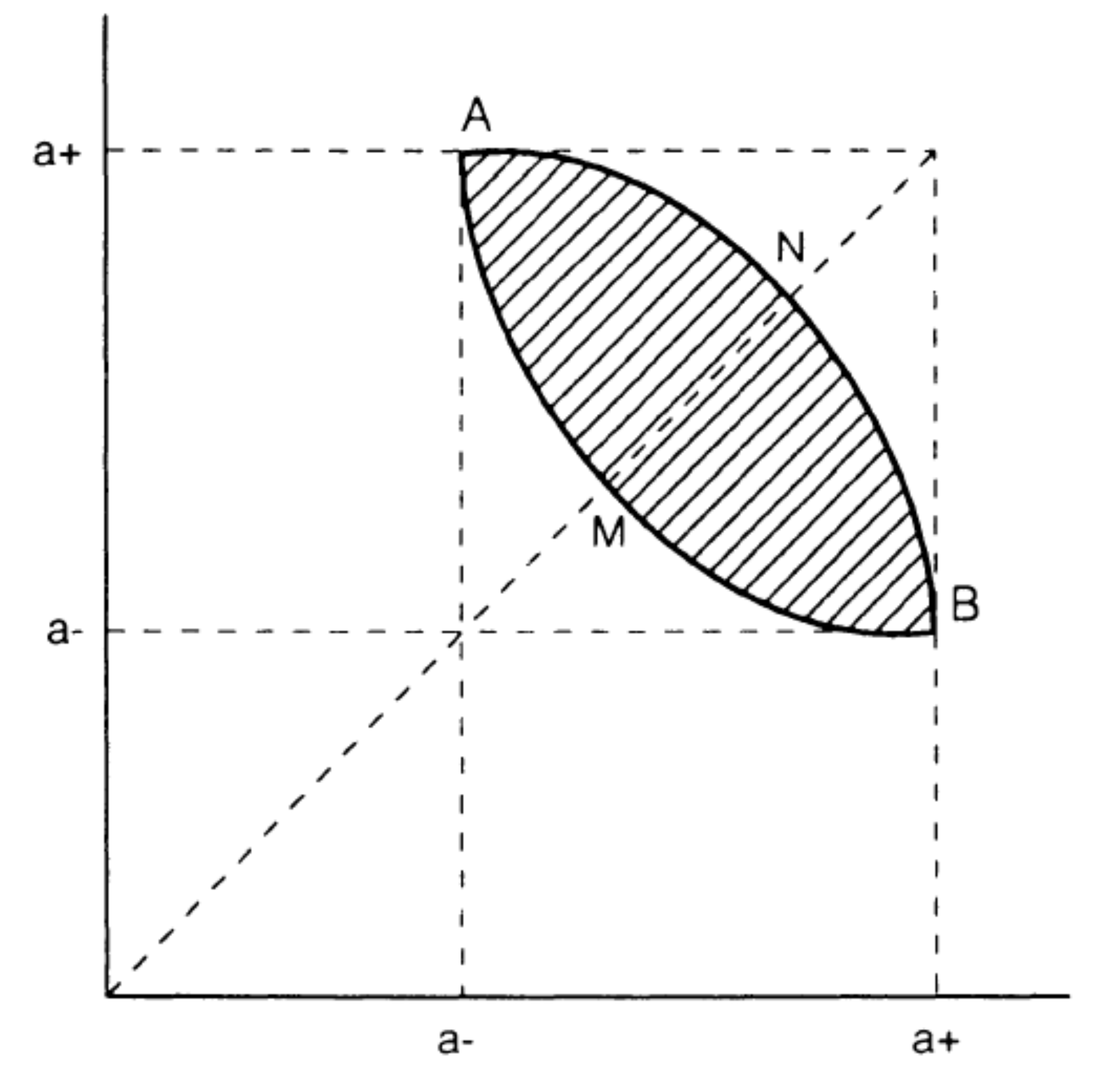}
  \vspace*{-10ex}
  \caption{Murat-Tartar Bounds.}
 \end{center}
\end{figure}
\begin{example}[Spherical Inclusions in two-phase medium]\label{si}
If $\omega= B(0,1)=\{ y\ | \ |y|\leq 1\} $ and
\begin{equation*}
\begin{aligned}
A_\omega(y)=a_B(r)I &= 
\left\{
\begin{array}{ll}
   \alpha I \quad\mbox{if } |y| \leq R,\\[1ex]
   \beta I \quad\mbox{if }  R < |y| \leq 1,
\end{array}\right.
\end{aligned}
\end{equation*}
$\alpha$ and $\beta$ are known as core and coating respectively. Then $A_\omega$ is equivalent to $\gamma I$, where $\gamma$ satisfies (see \cite{HS}): 
\begin{equation*}
\frac{\gamma - \beta}{\gamma + (N-1)\beta} = \theta \frac{\alpha - \beta}{\alpha + (N-1)\beta},\ \mbox{ where }\theta= R^N .
\end{equation*} 
\end{example}
\begin{example}[Elliptical Inclusions in two-phase medium]\label{ei}
For $m_1,..,m_N\in \mathbb{R}$ and $\rho + m_j >0 $ for $j=1,..,N$, the 
family of confocal ellipsoids $S_\rho$ of equation
$$ \sum_{j=1}^N \frac{y^2_j}{\rho + m_j} = 1,$$ defines implicitly a
real function $\rho$, outside a possibly degenerate ellipsoid in
a subspace of dimension $<$ $N$.\\
Now, if we consider $\omega=E_{\rho_2+m_1,..,\rho_2+m_N}= \Big\{ y\ | \ \sum\limits_{j=1}^N \frac{y^2_j}{\rho_2 + m_j} \leq 1\Big\},$ 
with $\rho_2 + \underset{j}{min}\ m_j >0 $  and  
\begin{align*}
 A_\omega(y)= a_E(\rho)I &=
 \left\{\begin{array}{ll} 
 \alpha I \quad\mbox{if } \rho \leq \rho_1,\\[1ex]
 \beta I \quad\mbox{if }  \rho_1 < \rho \leq \rho_2,
\end{array}\right.
\end{align*}
then $A_\omega$ is equivalent to a constant diagonal matrix $\Gamma= [\gamma_{jj}]_{1\leq j\leq N}$ satisfying
$$ \sum_{j=1}^N \frac{1}{\beta - \gamma_{jj}} =\ \frac{(1-\theta)\alpha + (N+\theta-1)\beta}{\theta \beta(\beta-\alpha)},\ \mbox{ where }\theta = \underset{j}{\Pi}\ \sqrt{\frac{\rho_1 + m_j}{\rho_2 + m_j}}.$$ 
\end{example}
\begin{remark}\label{SL}
In the problem of characterizing conductivities of 2-phase mixtures
taken in a given proportion, we know that, geometrically,
conductivities of all possible mixtures lie in a convex lens-shaped region 
bounded by an upper hyperbola and a lower hyperbola. Analytically, they satisfy
certain inequalities. The above Hashin-Shtrikman spherical inclusions 
appear among others as extremal structures $(M,N)$ (see Figure $1$) in this characterization 
in the sense that they are on the boundary of the above region, whereas the elliptical inclusions
appear on the all the points of two hyperbolas $AMB$ and $ANB$, respectively, expect the two intersection points $A$ and $B$.
The points $A$ and $B$ correspond the simple laminates which can be viewed 
as a limit of elliptical inclusions, where all of its axis goes to infinity except one axis
gives a strip structure of being infinite eccentricity.  
No wonder therefore that these structures have played important role in the construction of the above set.
\end{remark}
Before going to introduce the Bloch waves for the respective micro-structures 
here we address the interaction between adjacent elements among the micro-structures
and its importance towards fixing the state space for the Bloch waves. The main difference between periodic
and Hashin-Shtrikman structures lies in the choice of state space for ground state. Once the
right space is chosen, the procedure is quite similar to the one in the case of periodic structures except 
for technical details arising out of the difference in the state space. 
\subsection{An important comparison between periodic and Hashin-\\ Shtrikman micro-structures through its interfaces}\label{compare}

Let us start with some heuristics. Einstein principle says, 
mass creates curvature of space. Flat space with mass is equivalent
to curved space without mass. Mass is often compared with inhomogeneities
and curved space is compared with
inhomogeneous medium. Metric tensor is compared with variable coefficient
matrix. Laplace-Beltrami operator is compared with div-form operator with
variable coefficients.
We now consider a cover of the domain $\Omega$ by disjoint sets $\omega_p$ 
$(p\in\mathbb{N})$ up to a null set. If $\Omega$ lies in a flat Euclidean 
space, then the above open cover remains intact without any interaction 
between adjacent cells $\omega_p$, even though they touch each other.
Now, imagine that $\Omega$ is on curved space, say unit sphere. This amounts
to pulling the network of $\omega_p$ in a certain direction so that they
form a cover of the unit sphere. It is intuitively obvious that in order 
the cover be intact without collapsing/breaking, then necessarily, there has 
to be interaction between adjacent cells $\omega_p$. That is, adjacent balls 
exert force on each other. Analogously, if the balls are homogeneous, there 
is no interaction between them. On the other hand, if the balls represent 
inhomogeneity there has to be interaction between them.\\
\\
$\textbf{(1a):} \ $ As we see the periodic structures admit an invariance group, 
namely the additive group of integers which acts on the Euclidean
space by translation. Starting from one basic cell $Y$, we can tile 
the whole Euclidean space by its integer translates, i.e.
given $f=f(y)\in L^2 (Y)$, we can produce an element $F\in
L^2(\Omega)$ by periodically extending $f$ throughout $\mathbb{R}^N$. 
We say $F$ is $Y$-periodic extension of $f$.\\
$\textbf{(1b):} \ $ In the case of Hashin-Shtrikman structures, 
let us assume for simplicity the basic cell is a ball $B=B(0,1)$. Using translation group, we may move the
centers of the ball but does not generate a tiling of the whole
space. Similarly, the scaling group acts on the radius of the ball but again
it does not generate a tiling either. We need both the groups to generate
a tiling. Even though each one of the groups is commutative, but their
combination is not. We consider now a Vitali covering of $\Omega$
(up to a null set) by balls which are translated and scaled versions of 
the unit ball and denote them by $B_p= B(y^p,\epsilon_p)$. 
Let us imitate the previous construction. We start with $f=f(y)\in L^2(B)$.
We may define a translated and scaled version of $f$ on the ball $B_p$
denoted by $\widetilde{f}_p(x) = f(\frac{x-y^p}{\epsilon_p}),$ when $x\in B_p$.
Putting these versions together, we have a function 
$\widetilde{F}=\widetilde{F}(x)$ with $\widetilde{F}|_{B_p} = \widetilde{f}_p$ 
defined a.e. on $\Omega$. Then $\widetilde{F}\in L^2(\Omega)$.\\
\\
$\textbf{(2a):} \ $ Let us now consider Sobolev space $H^1$. Assume $f\in H^1(Y)$. 
It is classical that $F\in H^1(\Omega)$ if and only if
the trace of $f$ on the boundary of $Y$
are equal on opposite faces of the boundary. This later condition defines a
new space denoted by $H^1_{\#}(Y)$. This is one of the ways in which the space
$H^1_{\#}(Y)$ arises in periodic homogenization.\\
$\textbf{(2b):} \ $Going back to the construction in $\textbf{(1b)}$, we ask, 
under what condition, $\widetilde{F}\in H^1(\Omega)$ provided $f\in H^1(B)$?
Such questions are motivated by homogenization: starting from cell test
functions, we wish to construct oscillating approximation to physical
fields in $\Omega$ with finite energy and so they are usually in Sobolev
spaces. There is no doubt that $\widetilde{f}_p\in H^1(B_p),$ for each $p$. 
Since each ball $B_p$ touches its neighbors at a single point, 
it is clear that $\widetilde{F}$ is in $H^1$ in the open domain formed by the 
disjoint union of finitely many balls from the collection $B_p$.
Let us denote by $\widetilde{\Omega}$, the open set formed by the disjoint union of all
balls $B_p$.\\
\\
\textbf{Question:} Does $\widetilde{F}\in H^1(\widetilde{\Omega})$ ?\\
We may think that the answer is affirmative. This is wrong.
There is trouble because of the presence of small scales in the Vitali
covering. More precisely, $\nabla_x \widetilde{f}_p(x) = {\epsilon_p}^{-1}\nabla_yf(y).$
Since $\epsilon_p$ tends to zero, these negative powers pose
difficulties. In the case of periodic structures, the number of small cells
inside $\Omega$ is finite and so this difficulty does not arise.
In order to overcome this difficulty, we normalize $\widetilde{f}_p$ and define $f_p(x)$
by $\epsilon_p\widetilde{f}_p(x)$. As before, putting these $f_p$
together, we define $F=F(x)$ with $\widetilde{F}|_{B_p} = f_p$ a.e. on $\widetilde{\Omega}$.
This normalization does not alter the $L^2$-ness of $F$. But
now, computing the gradients, we have, $\nabla_x f_p(x) = \nabla_y f(y)$ on 
$B_p$ and there are no negative powers of $\epsilon_p$. Putting together these
gradients, we get an $L^2(\widetilde{\Omega})$ function. This proves $F$ is in $H^1(\widetilde{\Omega})$.\\
\\
\textbf{Question:} Is $F$ in $H^1(\Omega)$?\\
Answer: No, in general.
This answer may surprise a few, at first sight. One may think that, since
$\Omega$ and $\widetilde{\Omega}$ differ by a null set and since $F$ is in $H^1(\widetilde{\Omega})$,
the answer is affirmative. But this is wrong. In fact, the answer is affirmative
if and only if the trace of $f$ on the boundary of $B$ vanishes, i.e. $f\in H^1_0(B)$
This is again due to presence of infinite scales in the Vitali
covering: any ball in the covering has neighbors whose radii tend to zero.
Appearance of Dirichlet boundary condition is imposed by the Hashin-Shtrikman
micro-structure (see Appendix Lemma \ref{A1} for the proof ).
Thus, we see that the space $H^1_0(B)$ replaces $H^1_{\#}(Y),$
when we pass from periodic to Hashin-Shtrikman micro-structures.
\hfill
\qed
\\
\\
We will end this section with a brief survey of introducing Bloch waves in the periodic micro-structures.


\subsection{Survey of Bloch waves, Bloch eigenvalues and eigenvectors in periodic structures}\label{survey}

Introduction of Bloch waves in periodic structures is based on Floquet
principle: periodic structures can be regarded as multiplicative
perturbations of homogeneous media by periodic functions. This principle
gave rise to a new class of functions, namely $(\eta, Y)-$periodic functions.
\begin{align*}
\psi(.;\eta) \mbox{ is }(\eta;Y)-\mbox{periodic if }
\psi(y+2\pi m;\eta) = e^{2\pi i m.\eta}\psi(y;\eta) \quad \forall m \in \mathbb{Z}^N , y \in \mathbb{R}^N.
\end{align*}
Accordingly, the following spaces were used: $L^2_{\#}(\eta,Y),
L^2_{\#}(Y),H^1_{\#}(\eta,Y),H^1_{\#}(Y),$ etc. 
\begin{align*}
  L^2_{\#}(\eta;Y) &= \{ v\in L^2_{loc}(\mathbb{R}^N)\ |\   v \mbox{ is }(\eta,Y)-\mbox{periodic}\},\\
  H^1_{\#}(\eta; Y) &= \{ v\in H^1_{loc}(\mathbb{R}^N)\ |\ v \mbox{ is }(\eta,Y)-\mbox{periodic}\}.
\end{align*}
The $(\eta,Y)-$periodicity condition remains unaltered if we
replace $\eta$ by $(\eta + q)$ with $q \in\mathbb{Z}^N$, so $\eta$ 
can therefore be confined to the dual cell $\eta \in Y^{\prime} = [-\frac{1}{2},\frac{1}{2}]^N$
or equivalently dual torus. $\eta=0$ gives back the usual periodicity condition.\\
In fact, these classes are state spaces for Bloch waves.
The link between `Bloch waves' and the traditional `Homogenization theory' is as
follows: the homogenized tensor and the cell test functions can be
obtained as infinitesimal approximation from (ground state) Bloch waves at $\eta =0$
and its energy.\\
\\
We consider the operator 
\begin{equation*} \mathcal{A}_Y \equiv -\frac{\partial}{\partial y_k}\lb a^Y_{kl}(y)\frac{\partial}{\partial y_l} \rb \quad k,l=1,2..,N,\end{equation*}
where the coefficient matrix  $A_Y(y) = [a^Y_{kl}(y)]$ defined on $Y$ a.e., where $Y =[0,1]^N$ is known as the periodic cell
and $A_Y\in\mathcal{M}(\alpha,\beta;Y)$ for some $0<\alpha <\beta$, i.e. 
\begin{equation}a^Y_{kl}=a^Y_{lk}\hspace{5pt}\forall k,l \mbox{ and }\ (A_Y(y)\xi,\xi) \geq \alpha|\xi|^2,\ \ |A_Y(y)\xi|\leq \beta|\xi|\mbox{ for any } \xi \in \mathbb{R}^N,\mbox{ a.e. on }Y.\end{equation}
We define Bloch waves $\psi$ associated with the operator 
$\mathcal{A}_Y$ as follows. Let us consider the following spectral problem 
parameterized by $\eta \in \mathbb{R}^N$: \\
\\
Find $\zeta = \zeta(\eta) \in \mathbb{R}$ 
and $\psi_Y = \psi_Y(y; \eta)$ (not identically zero) such that
\begin{equation}\label{b1}
\mathcal{A}_Y\psi_Y(\cdot;\eta) =\zeta(\eta)\psi_Y(\cdot;\eta) \quad\mbox{in }\mathbb{R}^N,\ \ \ \psi_Y(\cdot;\eta) \mbox{ is }(\eta;Y)-\mbox{periodic.}
\end{equation}
By applying the Floquet principle, we define $\varphi_Y(y; \eta) = e^{-iy·\eta} \psi_Y(y; \eta)$  and then \eqref{b1} can be rewritten in terms of 
$\varphi_Y$ as follows:
\begin{equation}\label{ps}
\mathcal{A}_Y(\eta)\varphi_Y = \zeta(\eta)\varphi_Y \mbox{ in }\mathbb{R}^N, \quad \varphi_Y \mbox{ is $Y-$periodic,}
\end{equation}
where the operator $\mathcal{A}_Y(\eta)$  is defined by  
\begin{equation}\label{ae}
\mathcal{A}_Y(\eta) = -\Big(\frac{\partial}{\partial y_k} + i\eta_k\Big)\Big[a^Y_{kl}(y)\Big(\frac{\partial}{\partial y_l} + i\eta_l\Big)\Big].
\end{equation}
It is well known from \cite{CPV} that for each $\eta\in Y^{\prime}$ ,
the above spectral problem admits a discrete sequence of eigenvalues with the following properties:
\begin{equation*}
0 \leq \zeta_1 (\eta) \leq ... \leq \zeta_m (\eta) \leq ... \rightarrow \infty, \mbox{ and } \forall\ m \geq 1,\ \zeta_m(\eta)\mbox{ is a Lipschitz function of }\eta \in Y^{\prime}.
\end{equation*}
The corresponding eigenfunctions denoted by $\psi_{Y,m}(\cdot; \eta)$ and $\varphi_{Y,m} (\cdot; \eta)$ 
form orthonormal bases in the spaces of all $L^2_{loc} (\mathbb{R}^N )$
functions which are $(\eta; Y )-$periodic and $Y$-periodic, respectively. 
In fact, these eigenfunctions belong to the spaces $H^1_{\#}(\eta; Y)$ and $H^1_{\#}(Y),$ respectively.\\
\\
To obtain the spectral resolution of $\mathcal{A}^{\epsilon}_Y$ in an analogous manner, 
let us introduce Bloch waves at the $\epsilon$-scale:
$$\zeta^{\epsilon}_m(\xi) = \epsilon^{-2}\zeta_m(\eta),\quad \varphi^{\epsilon}_{Y,m}(x;\xi) = \varphi_{Y,m} (y;\eta),
\ \ \psi^{\epsilon}_{Y,m}(x;\xi) = \psi_{Y,m}(y; \eta)$$
where the variables $(x, \xi)$ and $(y;\eta)$ are related by $y = \frac{x}{\epsilon}$ and $\eta =\epsilon\xi$.
Observe that $\varphi^{\epsilon}_{Y,m}(x;\xi)$ is $\epsilon Y$-periodic (in $x$) and $\epsilon^{-1}Y^{\prime}$-periodic 
with respect to $\xi$. In the same manner, $\psi^{\epsilon}_{Y,m}(\cdot; \xi)$ is $(\epsilon\xi; \epsilon Y )-$periodic. 
The dual cell at $\epsilon$-scale is $\epsilon^{-1}Y^{\prime}$, where $\xi$ varies.\\
The functions $\psi_{Y,m}^{\epsilon}$ and $\varphi_{Y,m}^{\epsilon}$ 
(referred to as Bloch waves) enable us to describe the spectral resolution of $\mathcal{A}^{\epsilon}_Y$ 
(an unbounded self-adjoint operator in $L^2 (\mathbb{R}^N)$) in the orthogonal basis 
$\{e^{ix\cdot\xi} \varphi_{Y,m}^{\epsilon}(x; \xi)\ |\ m \geq 1, \xi \in \epsilon^{-1}Y^{\prime} \}$. 
More precisely, we have the following result:\begin{proposition}[Bloch decomposition \cite{CV}]\label{deA}
Let $g \in L^2(\mathbb{R}^N)$. The $m$-th Bloch coefficient of $g$ at the $\epsilon$-scale is
defined as follows:
\begin{equation}\label{bt}(B^{\epsilon}_m g)(\xi) = \int_{\mathbb{R}^N} g(x)e^{-ix\cdot\xi} \overline{\varphi^{\epsilon}}_{Y,m}(x;\xi)dx \quad \forall m \geq 1,\ \xi \in \epsilon^{-1}Y^{\prime}.\end{equation}
Then the following inverse formula holds:
\begin{equation}\label{ibt}g(x) = \int_{\epsilon^{-1}Y^{\prime}}(B^{\epsilon}_m g)(\xi)e^{ix\cdot\xi} \varphi^{\epsilon}_{Y,m}(x; \xi)d\xi.\end{equation}
And the Parseval's identity:
$$\int_{\mathbb{R}^N}|g(x)|^2 dx = \int_{\epsilon^{-1}Y^{\prime}}\sum_{m=1}^{\infty}|(B^{\epsilon}_m g)(\xi)|^{2} d\xi.$$
Finally, for all $g$ in the domain of $\mathcal{A}^{\epsilon}$, we have
$$\mathcal{A}^{\epsilon}g(x) = \int_{\epsilon^{-1}Y{\prime}}\sum_{m=1}^{\infty}\zeta^{\epsilon}_m(\xi)(B^{\epsilon}_m g)(\xi)e^{ix\cdot\xi} \varphi^{\epsilon}_{Y,m}(x;\xi)d\xi,$$
i.e. $\{ e^{ix\cdot\xi}\varphi_{Y,m}^{\epsilon}(x;\xi);\ m=1,\ldots,N, \ \xi\in\epsilon^{-1}Y^{\prime}\}$ is a basis for $L^2(\mathbb{R}^N)$. 
\hfill\qed\end{proposition}
\noindent 
Using the above proposition, the classical homogenization result was deduced in \cite{CV}. 
We consider a sequence $u^{\epsilon} \in H^1(\mathbb{R}^N)$ satisfying 
\begin{equation}\label{ed1}\mathcal{A}^{\epsilon}_Yu^{\epsilon} = f \quad\mbox{ in }\mathbb{R}^N, \end{equation} 
with the fact $u^{\epsilon}\rightharpoonup u$ in $H^{1}(\mathbb{R}^N)$ 
weak and $u^{\epsilon}\rightarrow u $ in $L^{2}(\mathbb{R}^N)$ strong.\\
The homogenization problem consists of passing to the limit, 
as $\epsilon \rightarrow 0$ in \eqref{ed1}, we get the homogenized equation 
satisfied by $u$, namely
\begin{equation*} \mathcal{A}^{*}_Yu =\ -\frac{\partial}{\partial x_k}\lb q_{kl}\frac{\partial u}{\partial x_l} \rb = f \quad\mbox{in }\mathbb{R}^N, \end{equation*}
where $A^{*}_Y=[q_{kl}]$ is the constant homogenized matrix \cite{A}.\\

Simple relation linking $A^{*}_Y=[q_{kl}]$ with Bloch waves is the following: 
$q_{kl}=\ \frac{1}{2}D^2_{kl}\zeta_1(0)$ (see \cite{CV}).
At this point, it is appropriate to recall that derivatives of the first eigenvalue
and eigenfunction at $\eta=0$ exist, thanks to the regularity property established 
in \cite{CV}. In fact, we know that there exists $\delta \geq 0$ such that the first 
eigenvalue $\zeta_1(\eta)$ is an analytic function on $B_{\delta}(0) = \{\eta\in\mathbb{R}^N \ |\ |\eta| < \delta\}$
and there is a choice of the first eigenvector $\varphi_{Y,1}(y;\eta)$ satisfying 
$$\eta \mapsto \varphi_{Y,1}(\cdot;\eta) \in H^1_{\#}(Y)\mbox{ is analytic on } B_{\delta}\quad\mbox{ and }\ \varphi_{Y,1}(y;0)= |Y|^{-1/2}.$$
\noindent
\\
We wish to carry out an analogous program for the generalized Hashin-Shtrikman structures.
\section{Bloch spectral analysis in the class of Hashin-Shtrikman micro-structures} 
\setcounter{equation}{0} 
\subsection{State space of Bloch Ground state on Hashin-Shtrikman structures}
What is the state space of Bloch waves on Hashin-Shtrikman structures? 
We assume Bloch waves on Hashin-Shtrikman structures still obey 
Floquet principle. As a consequence, because of
Dirichlet boundary condition (cf . [\textbf{(2b)}, Section\eqref{compare}])
we see that the state space for Bloch waves with momentum $\eta$ is
still $H^1_0(B)$ which is the state space for zero-momentum Bloch waves. 
There is no possibility of raising the energy of the ground state.
There is no effect of $\eta$ at all. To overcome the above difficulty, our proposal is the following: we keep
the Floquet principle intact and we change the state space from $H^1_0(B)$
to $H^1_c(B)$, which is the subspace of $H^1(B)$ whose boundary trace is a
constant (depending on the function). Compared with $H^1_0(B)$, the energy of
the ground state is now lowered. Since homogenization is a lower energy
approximation, one feels that $H^1_c(B)$ is more appropriate than
$H^1_0(B)$. Obviously, there is now the effect of $\eta$, which we will exploit
in the sequel. It should be remarked that this energy is obviously higher than
the ground energy in $H^1(B)$. Let us recall that the latter is the state
space for Neumann boundary condition. For various reasons, Bloch waves
with Dirichlet and Neumann boundary conditions do not yield desirable results. 
That is why, these boundary conditions are rejected.
\begin{remark}
Caution: By translation and scaling, we cannot produce a function $F$ in
$H^1(\Omega)$ starting from an element in $H^1_c(B)$.
\end{remark}
\begin{remark}
To get the motivation to fix the state space, we begin with the discussion by considering the ball $B$ 
as an example of basic-cell. One can retrace the same above arguments under the hypothesis `$A$ is equivalent 
to $M$' in a bounded open set $\omega\subset \mathbb{R}^N$ whenever non-uniform scales and translations are involved. 
\end{remark}
\noindent
We will be able to define a sequence of Bloch waves in the above state space
associated to the div-form operator on Hashin-Shtrikman structures. 
We are not able to show that they form an orthonormal basis diagonalizing the
operator under consideration. Surprisingly, we are able to show that the
associated ground state and its energy have some desired properties
described below.\\
More precisely, we show that one-half of the Hessian of the ground energy
is a scalar and it coincides with the homogenized coefficient of the 
Hashin-Shtrikman structures. In particular, this shows the usual spectral 
characterization holds for the homogenized coefficient of Hashin-Shtrikman
structures. Secondly, we can go beyond the homogenization approximation 
for the acoustic waves propagating on Hashin-Shtrikman micro-structured medium.
One knows that the situation in a general micro-structured medium is pretty 
complicated to describe beyond homogenization level. However, thanks to the 
introduction of Bloch waves, we can define the next order approximation beyond 
homogenization. This is achieved by the introduction of a certain 4-tensor 
``$d$''. This captures in a quantitative way the dispersion undergone by these
propagating waves. Thirdly, we make a link between this ``$d$'' and the similar 
$4$-tensor already introduced in the case of periodic structures and denoted 
by ``$d_Y$". This link enables us to prove a conjecture based on numerics 
concerning the behavior of ``$d_Y$" on periodic structures: 
``$d_Y$" attains its minimum value among all periodic Hashin-Shtrikman
micro-structures at periodic Apollo Hashin-Shtrikman structure \cite{TV2}. 
Perhaps, this last result provides solid justification of our
definition of Bloch waves on Hashin-Shtrikman structures.
\subsection{Bloch waves, Bloch eigenvalues and eigenvectors in the Hashin-Shtrikman structures}\label{nl9}

Let $\omega \subset \mathbb{R}^N$ be a bounded open with Lipschitz boundary and $A_\omega(y)=[a^\omega_{kl}(y)]_{1\leq k,l\leq N}
\in \mathcal{M}(\alpha,\beta,\omega)$. 
We consider the following spectral problem parameterized by $\eta \in \mathbb{R}^N$: Find $\lambda := \lambda(\eta) \in \mathbb{C}$ and $\varphi_\omega := \varphi_\omega(y; \eta)$ (not identically zero) such that
\begin{equation}\begin{aligned}\label{NE1}
-\Big(\frac{\partial}{\partial y_k} + i\eta_k\Big)\Big[a^\omega_{kl}(y)(\frac{\partial}{\partial y_l} + i\eta_l)\Big]\varphi_\omega(y;\eta) &= \lambda(\eta)\varphi_\omega(y;\eta) \mbox{ in }\omega, \\
\varphi_\omega(y;\eta)\mbox{ is constant on }\partial\omega,  \\
\int_{\partial\omega} a^\omega_{kl}(y)(\frac{\partial}{\partial y_l} + i\eta_l)\varphi_\omega(y;\eta)\nu_k\ d\sigma &= 0. 
\end{aligned}\end{equation}
(where $\nu$ is the outer normal unit vector on the boundary and $d\sigma$ is the surface measure on $\partial\omega$).
\paragraph{Weak formulation:}
Here first we introduce the function spaces
\begin{align*}
L^2_c(\omega) &=  \ \{\varphi\in L^2_{loc}(\mathbb{R}^N) \  | \ \varphi \mbox{ is constant in } \mathbb{R}^N \smallsetminus\omega \},\\
H^1_c(\omega) &=\ \{\varphi\in H^1_{loc}(\mathbb{R}^N) \  | \ \varphi \mbox{ is constant in } \mathbb{R}^N \smallsetminus\omega \},\\
              &= \ \{\varphi\in H^1(\omega) \  | \ \varphi|_{\partial\omega} = \mbox{constant} \}.
\end{align*}
Here “$c$” is a floating constant depending on the element under consideration.\\
Notice that $L^2_c(\omega)$ and $H^1_c(\omega)$ are proper subspace of $L^2(\omega)$ and  $H^1(\omega)$ respectively, and they inherit
the subspace norm-topology of the parent space. \\
In a similar fashion, we define for $\eta\in \mathbb{R}^N$:
\begin{align*}
&L^2_{c}(\eta;\omega) = \ \{\varphi \in L^2_{loc}(\mathbb{R}^N) \  | \ e^{-iy\cdot \eta}\varphi \mbox{ is constant in } \mathbb{R}^N \smallsetminus\omega \},\\
&H^1_c(\eta;\omega) =\ \{\varphi\in H^1_{loc}(\mathbb{R}^N) \  | \ e^{-iy\cdot \eta}\varphi \mbox{ is constant in } \mathbb{R}^N \smallsetminus\omega \}.
\end{align*}
The next step is to give a weak formulation of the problem in these function spaces. To this
end, let us introduce some bilinear forms:
\begin{align*}
a_\omega(u,w) &= \ \int_\omega a^\omega_{kl}(y)\frac{\partial u}{\partial y_l}\overline{ \frac{\partial w}{\partial y_k}}dy,\\
a_\omega(\eta)(u,w) &= \ \int_\omega a^\omega_{kl}(y)\lb\frac{\partial u}{\partial y_l} +i\eta_l v \rb \overline{ \lb \frac{\partial w}{\partial y_k} + i\eta_k w \rb }dy,\\
(u,w) &= \ \int_\omega u\overline{w} dy. 
\end{align*}
Based on these above bilinear forms we are interested into proving the existence 
of the eigenvalue and the corresponding eigenvector $(\lambda(\eta), \varphi(y;\eta))$
with $\lambda(\eta)\in\mathbb{C}$ and $\varphi(.;\eta) \in H^1_c(\omega)$ of the following weak formulation of  \eqref{NE1}:
\begin{equation}\label{nl17}
a_\omega(\eta)(\varphi_\omega(y;\eta), \psi) = \lambda(\eta)(\varphi_\omega(y;\eta),\psi) \quad \forall \psi\in H^1_c(\omega).
\end{equation}
\paragraph{Existence Result:}
We start with the following parameterized boundary value problem:
\\
\\
Given $F\in L^2_c(\omega)$, find $u\in H^1_c(\omega)$ satisfying 
\begin{equation}\label{fe1}
 a_\omega(\eta)(u,v) = (F,v) \quad \forall v \in H^1_c(\Omega).
\end{equation}
As we see that for each $\eta $ fixed, the bilinear form 
\begin{equation*}a_\omega(\eta) : H^1_c(\omega) \times H^1_c(\omega) \rightarrow \mathbb{C} \quad\mbox{is continuous.}\end{equation*}
And if $|\eta|$ is bounded, then for some constant $C$ large enough, it satisfies the following ellipticity property of 
\begin{equation}\label{ellip} \ a_\omega(\eta)(v,v) + C||v||^2 \geq \ \frac{\alpha}{2} \{ ||\nabla v||^2 + ||v||^2  \} \quad \forall v \in H^1_c(\omega).\end{equation}
The continuity part follows trivially, the ellipticity follows as
\begin{equation*}\begin{aligned}
a_\omega(\eta)(v,v) = &\ \int_{\omega} A_\omega(y)(\nabla + i\eta)v\cdot\overline{(\nabla + i\eta)v}\ dy\\
                               \geq & \ \alpha (||\nabla v||^2 + |\eta|^2||v||^2) - 2\beta||\nabla v||\cdot|\eta|||v|| \\
                               \geq & \ \alpha (||\nabla v||^2 + |\eta|^2||v||^2) - \beta\lb \frac{\alpha}{2\beta} ||\nabla v||^2 + \frac{2\beta}{\alpha}|\eta|^2||v||^2 \rb, 
\end{aligned}\end{equation*}
where in the above step we use the Cauchy-Schwarz inequality. Then choosing a constant $C \geq \frac{\alpha}{2} + (\frac{2\beta^2}{\alpha} - \alpha)|\eta|^2$, we get \eqref{ellip}.  \\
Thus, whenever $|\eta|$ is bounded, we can apply Lax-Milgram Lemma and get the solvability of \eqref{fe1}.
To solve the corresponding spectral problem \eqref{nl17}, we consider the corresponding Green’s operator
\begin{align*}
&G_\omega(\eta) : L^2_c(\omega) \rightarrow H^1_c(\omega) \hookrightarrow L^2_c(\omega),\\ 
&G_\omega(\eta)F =\ u. 
\end{align*}
Since the inclusion $H^1_c(\omega) \hookrightarrow L^2_c(\omega)$ is compact, $G_\omega(\eta)$ is a compact operator whenever $|\eta|$ is bounded
and further $G_\omega(\eta)$ is self-adjoint which essentially comes from the symmetry of $a^\omega_{kl} = a^\omega_{lk}$. Thus, applying the spectral theory
of compact self-adjoint operators, we arrive at:
\begin{theorem}\label{deH}
Fix $\eta \in \mathbb{R}^N$. Then there exist a sequence of eigenvalues $\{ \lambda_m(\eta); m \in \mathbb{N} \}$
and corresponding eigenvectors $\{ \varphi_{\omega,m} (y;\eta)\in H^1_c(\omega), m \in \mathbb{N} \}$ such that
\begin{align*}
&(i) \ \mathcal{A}_\omega(\eta)\varphi_{\omega,m}(y;\eta) = \lambda_m(\eta)\varphi_{\omega,m}(y;\eta) \quad \forall m \in \mathbb{N}.\\ 
&(ii) \ 0 \leq \lambda_1(\eta) \leq \lambda_2 (\eta) \leq ... \rightarrow \infty. \ \mbox{ Each eigenvalue is of finite multiplicity. }\\
&(iii)\ \{\varphi_{\omega,m}(\cdot;\eta);\ m \in  \mathbb{N} \}\ \mbox{ is an orthonormal basis for }L^2_c(\omega).\\
&(iv)\ \mbox{ For }\varphi \mbox{ in the domain of } \mathcal{A}_\omega(\eta), \mbox{ we have }
\mathcal{A}_\omega(\eta)\varphi(y) = \sum_{m=1}^{\infty} \lambda_m (\eta)(\varphi, \varphi_{\omega,m} (\cdot;\eta))\varphi_{\omega,m} (y;\eta).
\end{align*}
\end{theorem}
\noindent
The above result establishes the existence of Bloch eigenvalues and Bloch waves and describes
some of their properties. 
\begin{remark}
 In context to periodic structures, we do not have $\mathcal{A}_\omega = \int \mathcal{A}_\omega(\eta)d\eta$.
 As a consequence, we do not have a basis of eigenvectors. Fortunately, this is not required 
 for homogenization purposes. What is required is some information  about ground state.
\end{remark}

\subsection{Regularity of the Ground State}

We have obtained a simplified picture of the operator $\mathcal{A}_\omega$, namely it
is a multiplication operator with eigenvalues $\{\lambda_m (\eta)\}_{m\in \mathbb{N}}$
as multipliers. We expect, the higher eigenvalues $\{\lambda_m (\eta); m \geq 2\}$
do not play any role as they are not excited in the homogenization process. 
Thus, we are reduced to consider the ground state  $\varphi_1(y;\eta)$
and the corresponding energy $\lambda_1(\eta)$ and what matters in homogenization is
their regular behavior near $\eta = 0$. In the sequel, we establish two types of results which are consistent with 
the regularity observed in the homogeneous case.
The first one is a global regularity result valid for all $m \geq 1$. The second one is
local regularity of the ground state at $\eta = 0.$
Our first result is concerned with global regularity of all eigenvalues.
\begin{proposition}
 For $m \geq 1 $, $\lambda_m(\eta)$ is a Lipschitz function of $\eta$.
\end{proposition}
\begin{proof}
According to Courant-Fischer characterization of the eigenvalue via min-max principle, we have
$$\lambda_m(\eta) = \underset{dim F = m}{min}\ \underset{v\in F}{max}\ \ \frac{a_\omega(\eta)(v,v)}{(v,v)},$$
where $F$ ranges over all subspace of $H^1_{c}(\omega)$ of dimension=$m$. We notice that
$a_\omega(\eta)(v,v)$ can be decomposed as 
$$ a_\omega(\eta)(v,v) = a_\omega(\eta^{\prime})(v,v) + R_\omega(v,\eta,\eta^{\prime}),$$
where
$$ R_\omega =  \int_\omega a^\omega_{kl}(\eta_l\eta_k-\eta^{\prime}_l\eta^{\prime}_k)|v|^2 dy
+ \int_\omega a^\omega_{kl} \frac{\partial v}{\partial y_l}\overline{(i\eta_k-i\eta^{\prime}_k)v}dy 
+ \int_\omega a^\omega_{kl} (i\eta_l-i\eta^{\prime}_l)v \overline{\frac{\partial v}{\partial y_k}} dy. 
$$
By Cauchy-Schwartz inequality, $R_\omega$ can be estimated as 
$$ |R_\omega| \leq c |\eta -\eta^{\prime}| ||v||_{H^1_c(\omega)}. $$
And consequently, using the above min-max characterization, we get 
$$ |\lambda_m(\eta) -\lambda_m(\eta^{\prime})| \leq c_m |\eta - \eta^{\prime}|.$$
\hfill\end{proof}
\noindent
For homogenization purposes, above global regularity is not sufficient. We need a strong
local regularity of the ground state and the corresponding energy. 
Here, as an application of the Kato-Rellich perturbations theory \cite{K},
we will eastablish the analyticity of $(\lambda_1(\eta), \varphi_{\omega,1}(y;\eta))$ 
in some neighborhood $\omega^{\prime}$ of $\eta = 0$. We execute it, in the following steps:

\begin{proposition}
$\mathcal{A}_\omega(\eta)$ is a holomorphic family of type $(B)$.
\end{proposition}
\begin{proof}
For the definition purpose we start with this simple extension 
$\mathcal{A}_\omega(\eta^{\prime})$ of $\mathcal{A}_\omega(\eta)$ to $\eta^{\prime}\in \mathbb{C}^N$,
which will be shown as a holomorphic family of type $(B)$. 
We refer \cite{K} for the definition of a holomorphic family of sesquilinear forms. 
Corresponding to a sesquilinear form $\mathbf{t}_\omega$ on $L^2_c(\omega)$ with domain $\mathcal{D}(\mathbf{t}_\omega) = \ H^1_c(\omega)$, 
we define the quadratic form $\mathbf{t}_\omega[u]$ by
\begin{center}$\mathbf{t}_\omega[u] := \mathbf{t}_\omega[u, u].$\end{center}
We denote the real and imaginary parts of $\mathbf{t}_\omega[u]$ by $\Re \mathbf{t}_\omega[u]$ and $\Im \mathbf{t}_\omega[u],$ respectively.
The sesquilinear form associated with the operator
$\mathcal{A}_\omega(\eta)$, namely, the family of sesquilinear forms $\mathbf{t}_\omega(\eta^{\prime})$
depending on $\eta^{\prime}\in \mathbb{C}^N$,
with same domain given by $D(\mathbf{t}_\omega(\eta^{\prime})) = H^1_c(\omega)$ contained in $L^2_c(\omega)$ defined as 
\begin{equation}\label{st}
\mathbf{t}_\omega(\eta^{\prime})[u, w] = \int_\omega A_\omega(y)(\nabla + i\eta^{\prime})u\cdot\overline{(\nabla +i\eta^{\prime})w}\ dy.
\end{equation}
$\mathbf{t}_\omega(\eta^{\prime})[u, w]$ reduces to the sesquilinear form corresponding to the operator $\mathcal{A}_\omega(\eta),$ when $\eta = \eta^{\prime}\in \mathbb{R}^N$. 
Since $\eta\in \omega^{\prime}$ a neighborhood around zero, we restrict $\eta$ to the set
\begin{equation*} R:= \Big\{\eta^{\prime} \in \mathbb{C}^N : \eta^{\prime} = \sigma + i\tau ; \sigma, \tau \in \mathbb{R}^N, |\sigma| < \frac{1}{2}, |\tau| < \frac{1}{2} \Big\}.\end{equation*}
\paragraph{Step 1.$\hspace{2pt}$ Sectoriality of $\hspace{3pt}\mathbf{t}_\omega(\eta^{\prime})$:} 
The sectoriality of $\mathbf{t}_\omega(\eta^{\prime})$ means that the 
numerical range $\mathbf{t}_\omega(\eta^{\prime})[u],$ as $u$ varies on the unit sphere
in $L^2_c(\omega),$ lies inside a sector of the complex plane with vertex 
at some point in the complex plane, and its projection on the real axis 
is bounded from below. However, the real part of $\mathbf{t}_\omega(\eta^{\prime})[u]$ is 
\begin{equation*}
\Re\mathbf{t}_\omega(\eta^{\prime})[u] = \mathbf{t}_\omega(\eta^{\prime})[u]=  \int_{\omega} A_\omega(y)(\nabla + i\eta^{\prime})u\cdot\overline{(\nabla + i\eta^{\prime})u}\ dy.
\end{equation*}
Similar to \eqref{ellip}, by choosing a constant $C \geq \frac{\alpha}{2} + (\frac{2\beta^2}{\alpha} - \alpha)|\eta^\prime|^2$, we have   
\begin{equation}\label{ret}
\mathbf{t}_\omega(\eta^{\prime})[u] + C||u||^2_{L^2_c(\omega)} \geq \frac{\alpha}{2}||u||^2_{H^1_c(\omega)}.
\end{equation}
Following this, we now consider a new family of forms $\mathbf{t}_\omega^{\prime}(\eta^{\prime})$,
with same domain as that of the family $\mathbf{t}_\omega(\eta^{\prime}),$ namely $H^1_c(\omega)$, defined by
\begin{equation*} \mathbf{t}_\omega^{\prime}(\eta^{\prime})[u, w] = \mathbf{t}_\omega(\eta^{\prime})[u, w] + C(u, w)_{L^{2}(\omega)}.\end{equation*}
For the new family $\mathbf{t}_\omega^{\prime}(\eta^{\prime})$, the inequality \eqref{ret} reads as
\begin{equation*} 
\mathbf{t}_\omega^{\prime}(\eta^{\prime})[u,u] \geq \frac{\alpha}{2}||u||^{2}_{H^1_c(\omega)}.
\end{equation*}
Thus, $\mathbf{t}_\omega^{\prime}(\eta^{\prime})$ is sectorial for each $z$ and uniformly with respect to $z\in R$. 
Since the addition of a scalar does not affect the sectorial nature, it follows that $\mathbf{t}_\omega(\eta^{\prime})$ is sectorial.
\paragraph{Step 2.$\ \ \mathbf{t}_\omega(\eta^{\prime})$ is closed:}
Let $u_n \in H^{1}_c(\omega)$ be such that 
$u_n \rightarrow u$. By the definition of
$\mathbf{t}_\omega^{\prime}$-convergence, 
we have $\mathbf{t}_\omega^{\prime}[u_n - u_m] \rightarrow 0$ 
and in turn  $\Re\mathbf{t}_\omega^{\prime}[u_n - u_m ]\rightarrow 0$ which implies $||u_n - u_m||^2_{H^{1}(\omega)}\rightarrow 0$,
as $ n, m \rightarrow \infty. $
Since $H^1_c(\omega)$ is complete, thus by following the definition of $\mathbf{t}_\omega^{\prime}$-convergence, 
clearly we have $\mathbf{t}_\omega^{\prime}(\eta^{\prime})$ is closed for each $z$. Due to the fact
that adding a real number is independent of the property of closedness of the form, 
it follows that $\mathbf{t}_\omega(\eta^{\prime})$ is closed.
\paragraph{Step 3.$\ \ \mathbf{t}_\omega(\eta^{\prime})$ is a holomorphic family of type (a):}
It comes from the fact that $\mathbf{t}_\omega(\eta^{\prime})[u]$
is a quadratic polynomial in $\eta^{\prime}$ for each fixed $u \in H^{1}_c(\omega)$.\\
\\ 
Now following \cite[Page no. 322]{K} associated with each 
$\mathbf{t}_\omega(\eta^{\prime})$, there exists a unique $m$-sectorial 
operator $\mathcal{A}_\omega(\eta^{\prime})$ with domain contained in 
$\mathcal{D}(\mathbf{t}_\omega)(\eta^{\prime}) $ ($H^1_c(\omega)$)
and the family of such operators associated with a holomorphic family of
sesquilinear forms of type $(a)$ is called a holomorphic family of type $(B)$. 
\hfill
\end{proof}
\paragraph{Application of Kato-Rellich theorem:}
We now apply the following Kato--Rellich theorem 
in order to show the existence of analytic eigenvalue 
$\lambda_1(\eta^{\prime})$ and analytic eigenvector $\varphi_{\omega,1}(\eta^{\prime})$ branches
with values in $L^2(\omega)$, near $\eta^{\prime}= 0$.
\begin{theorem}(see \cite{K})
Let $\mathcal{A}_\omega(\eta^{\prime})$, for $\eta^{\prime}$ in a domain 
in $\mathbb{C}^N$, be a holomorphic family of type $(B)$ with 
domains $\mathcal{D}(\mathcal{A}_\omega(\eta^{\prime}))$ contained in $L^2_c(\omega)$. 
Further, let $\lambda_0$ be an isolated
eigenvalue of $\mathcal{A}_\omega(\eta^{\prime}_0)$ that is algebraically simple. 
Then there exists a neighborhood 
$R_{0}\subseteq \mathbb{C}^N$ of $\eta^{\prime}_0$ such that 
for $\eta^{\prime}$ in $R_0$ , the following affirmations hold:
\begin{enumerate}
\item  There is exactly one point $\lambda_1(\eta^{\prime})$ of $\sigma(\mathcal{A}_\omega(\eta^{\prime}))$ 
near $\lambda_0$. Also, $\lambda_1(\eta^{\prime})$ 
is isolated and algebraically simple. 
Moreover, $\lambda_1(\eta^{\prime})$ is an analytic function of $\eta^{\prime}$.
\item  There is an associated eigenvector $\varphi_{\omega,1}(\eta^{\prime})$ depending 
analytically on $\eta^{\prime}$ with values in $L^2_c(\omega)$.
\end{enumerate}
\end{theorem}
\noindent
We choose $\eta^{\prime} = 0$ and prove the required non-degeneracy of 
the eigenvalue $\lambda_1(0)$ of $\mathcal{A}_\omega(0)$ in the following proposition.
\begin{proposition}
Zero is an eigenvalue of $\mathcal{A}_\omega(0)$ and is an isolated point of the spectrum
with its algebraic multiplicity one.
\end{proposition}
\begin{proof}
Consider the problem
\begin{equation*}
 \begin{aligned}
& -\frac{\partial}{\partial y_k}(a^{\omega}_{kl}(y)\frac{\partial}{\partial y_l})\varphi_{\omega}(y) = \lambda \varphi_{\omega}(y) \mbox{ in } \omega,\\ 
&\varphi_{\omega}(y)\in H^1_c(\omega) \quad\mbox{and}\quad \int_{\partial\omega} a^{\omega}_{kl}(y)\frac{\partial}{\partial y_k}\varphi_{\omega}(y)\cdot\nu_l\ d\sigma = \ 0.
\end{aligned}
\end{equation*}
$\lb-\frac{\partial}{\partial y_k}(a^{\omega}_{kl}(y)\frac{\partial}{\partial y_l})\rb^{-1}$ is a compact 
self-adjoint non-negative definite operator from $L^2(\omega)$ to itself. 
When $\lambda =0$, $\varphi_{\omega}=$ constant is the only solution (via integration by parts).
Which says the first eigenvalue $\lambda_1$ is zero, it is an simple eigenvalue
with the corresponding eigenspace $E_{\lambda_1} =\{$set of all constants$\}$.
\hfill
\end{proof}
\noindent
Therefore, following the above stated Kato-Rellich theorem by considering $\eta^{\prime}=0$  
and $\lambda_1(0)=0$ and $\varphi_{\omega,1}(y;0) = |\omega|^{-1/2}$, there exists a neighborhood of $R$, say $R_0$, where there is an eigenvalue branch $\lambda_1(\eta^{\prime})$ and the corresponding 
eigenvector branch $\varphi_{\omega,1}(\cdot;\eta^{\prime})$, which are analytic with values 
in $\mathbb{C}$ and $H^1_c(\omega)$ respectively.
\begin{remark}
The strong holomorphy for a Banach space valued functions of a complex variable 
which is same as the notion of weak holomorphy, i.e. for each $v \in L^2_c(\omega)$, 
the function $\langle \varphi(\eta^{\prime}), v)\rangle_{L^2_c(\omega)}$ is 
holomorphic in $R_0$ and using the usual duality argument, together with the 
 compact embedding $ H^1_c(\omega) \hookrightarrow L^2_c(\omega) \hookrightarrow H^{-1}(\omega)$,
$\langle \varphi(\eta^{\prime}),v\rangle_{H^1_c(\omega), H^{-1}(\omega)}$ 
is holomorphic in $R_0$, for each $v \in L^2_c(\omega)$.
\end{remark}
\noindent
We restrict $\lambda_1(\eta^{\prime})$ and $\varphi_{\omega,1}(\cdot;\eta^{\prime})$ 
to $\eta^{\prime}$ real, we will be using the symbol $\eta\in\mathbb{R}^N$ for that
and we denote neighborhood $\omega^{\prime}$ of $\eta =
0$ in $\mathbb{R}^N$ containing the real-analytic eigenelements.\\

In the following section we will explore the analyticity of the ground state and compute various derivatives.

\section{Computation of derivatives of ground state}\label{nl15}
\setcounter{equation}{0} 
Here we compute the derivatives of $\lambda_1(\eta)$ and 
derivatives of $\varphi_{\omega,1}(\cdot;\eta)$ at $\eta = 0$.
We will see that $2M$ ($M$ is the homogenized matrix)
coincides with the Hessien matrix of $\lambda_1(\eta)$ at $\eta = 0$.
Before proceeding further, there is a need for proper normalization of 
ground state and this is what we do next.

\paragraph{Boundary Normalization:}
Since $\varphi_{\omega,1}(y;\eta)\in H^1_c(\omega)$ i.e. $\varphi_{\omega,1}(y;\eta)|_{\partial\omega} = c_\eta \in \mathbb{C}$, where 
some constants which may depend upon $\eta$.
There is a choice of eigenvectors $\varphi_{\omega,1}(y; \eta)$ of dimension $1$
that depends analytically on $\eta$ in a small neighborhood $\omega^{\prime}$ of $\eta = 0$.
At $\eta =0$ we already made a choice that $\varphi_{\omega,1}(y;0) = |\omega|^{-1/2}$. Now
due to the analyticity of $\varphi (y; \eta )$ near $\eta=0$, we choose a
neighborhood (still denoting as $\omega^{\prime}$), where $\varphi_{\omega,1}(y; \eta )$ is always non-zero, in particular $c_{\eta} \neq 0$. 
Therefore, by multiplying $\varphi_{\omega,1}(y;\eta )$ by $\frac{ |\omega|^{-\frac{1}{2}}}{c_\eta}$, here we make a 
new choice of $\varphi_{\omega,1}(y;\eta)$ which is uniformly (w.r.t. $\eta$) constant on the  boundary, i.e.
$$\varphi_{\omega,1}(y;\eta)|_{\partial\omega} = |\omega|^{-1/2}\quad\mbox{for } \eta \in \omega^{\prime}.$$ 
Consequently, for any $l\in \mathbb{Z}^N_{+}$ ($l\neq 0$),
\begin{equation*}D^{l}_{\eta}\varphi_{\omega,1}(y;0) \in H^1_0(\omega). \end{equation*} 

\paragraph{Derivatives of $\lambda_1(\eta)$ and $\varphi_{\omega,1}(\eta)$ at $\eta=0$:}
The procedure consists of differentiating the following eigenvalue equation
\begin{align}
&-\Big(\frac{\partial}{\partial y_k} + i\eta_k\Big)\Big[a^{\omega}_{kl}(y)\Big(\frac{\partial}{\partial y_l} + i\eta_l\Big)\Big]\varphi_{\omega,1}(y;\eta) = \lambda_1(\eta)\varphi_{\omega,1}(y;\eta) \mbox{ in }\omega, \label{com}\\
&\varphi_{\omega,1}(y;\eta)= |\omega|^{-1/2} \mbox{ on }\partial\omega \quad\mbox{and}\quad \int_{\partial\omega} a^{\omega}_{kl}(y)\Big(\frac{\partial}{\partial y_k} + i\eta_k\Big)\varphi_{\omega,1}(y;\eta)\cdot\nu_l\ d\sigma = \ 0. \label{comb}
\end{align}
We begin with the expression \eqref{ae} of the shifted operator
\begin{align*}
\mathcal{A}_\omega(\eta) &=\ \mathcal{A}_\omega + i\eta_k C^{\omega}_k + \eta_k\eta_l a^{\omega}_{kl}(y),\\
\mbox{with }\ C^{\omega}_k(\varphi) &=\ - a^{\omega}_{kj}(y)\frac{\partial \varphi}{\partial y_j} - \frac{\partial}{\partial y_j}(a^{\omega}_{kj}(y)\varphi).
\end{align*}
\paragraph{Step 1.\hspace{2pt} Zeroth order derivatives:}
We simply recall that $\varphi_{\omega,1}(y;0) = |\omega|^{-1/2}$ by our choice and $\lambda_1(0) = 0.$
\paragraph{Step 2.\hspace{2pt} First order derivatives of $\lambda_1(\eta)$ at $\eta =0$:}
By differentiating the equation \eqref{com} once with respect to $\eta_k$, we obtain
\begin{equation}\label{fto}D_k(\mathcal{A}_\omega(\eta)- \lambda_1(\eta))\varphi_{\omega,1}(\cdot;\eta) + (\mathcal{A}_\omega(\eta) - \lambda_1(\eta))D_k\varphi_{\omega,1}(\cdot;\eta) = 0.\end{equation}
Taking scalar product with $\varphi_{\omega,1}(\cdot;\eta)$ in $L^2(\omega)$ and evaluate the above relation at $\eta =0$, we get
\begin{equation*}\langle D_k (\mathcal{A}_\omega(0) - \lambda_1(0))\varphi_{\omega,1}(\cdot;0),\varphi_{\omega,1} (\cdot;0)\rangle = 0.\end{equation*}
Since $D_k \mathcal{A}_\omega(0) = iC^{\omega}_k $ and $\varphi_{\omega,1}(\cdot;0) = |\omega|^{-1/2}  $ does not depend on $y$, then we obtain
\begin{equation*}C^{\omega}_k \varphi_{\omega,1}(\cdot;0) = -\frac{\partial}{\partial y_j}(a^{\omega}_{kj}\varphi_{\omega,1}(\cdot;0)).\end{equation*}
Hence its integral over $\omega$ vanishes through integration by parts together with using \eqref{comb}.
It follows therefore that 
\begin{equation}\label{l10}  D_k \lambda_1(0) = 0 \quad \forall\hspace{2pt} k = 1,..,N. \end{equation}
\paragraph{Step 3. \hspace{2pt} First order derivatives of $\varphi_{\omega,1}(.;\eta)$ at $\eta=0$: }
Using \eqref{l10} and from \eqref{fto} at $\eta = 0$, we get the following equation 
\begin{equation*} \mathcal{A}_\omega D_k \varphi_{\omega,1}(\cdot;0) =- D_k \mathcal{A}_\omega(0)\varphi_{\omega,1}(\cdot;0)\end{equation*}
which can be written as
\begin{equation}\label{p10}
\mathcal{A}_\omega D_k\varphi_{\omega,1}(.;0) = -iC^{\omega}_k\varphi_{\omega,1}(.;0) = i\varphi_{\omega,1}(.;0)\frac{\partial}{\partial y_j}(a^{\omega}_{kj})\quad\mbox{in }\omega.
\end{equation}
Differentiating the boundary condition in \eqref{comb} with respect to $\eta_k$ at $0$, we get
\begin{align}
&D_k\varphi_{\omega,1}(\cdot;0) =\ 0  \quad\mbox{on }\partial\omega \label{p10b}\\
\mbox{and}\quad \int_{\partial\omega} A_\omega(y) & \lb \nabla_y  D_k\varphi_{\omega,1}(y;0) + i\varphi_{\omega,1}(y;0)e_k\rb\cdot\nu \ d\sigma = 0 .\label{p10bb}
\end{align}
As we can see along with boundary condition \eqref{p10b} for the elliptic equation \eqref{p10},
the solution $D_k\varphi_{\omega,1}(\cdot;0)$ gets uniquely determined. So, comparing \eqref{p10}, \eqref{p10b} with 
the test function $w_{e_k}(y)$ defined in \eqref{hsp}, we get
\begin{equation*}
D_k \varphi_{\omega,1} (y; 0) = i|\omega|^{-1/2}(w_{e_k}(y)- y_k).
\end{equation*}
And using the fact $A$ is equivalent to $M$, we have
$$ A_\omega(y)(\nabla_y D_k\varphi_{\omega,1}(y;0) + i\varphi_{\omega,1}(y;0)e_k)\cdot \nu = i|\omega|^{-1/2} A_\omega(y)\nabla w_{e_k}(y)\cdot \nu =\ i|\omega|^{-1/2}\ Me_k\cdot\nu. $$
Hence, it  satisfies \eqref{p10bb}, i.e. 
$$\int_{\partial\omega} A_\omega(y)  \lb \nabla_y  D_k\varphi_{\omega,1}(y;0) + i\varphi_{\omega,1}(y;0)e_k\rb \cdot\nu \ d\sigma= \int_{\partial\omega}i|\omega|^{-1/2}\ M e_k\cdot \nu\ d\sigma = 0 . $$
In particular, $D_k \varphi_{\omega,1} (y; 0)$ is purely imaginary.
\paragraph{Step 4.\hspace{2pt} Second derivatives of $\lambda_1(\eta)$ at $\eta=0$:}
We differentiate \eqref{fto} with respect to $\eta_l$ to obtain
\begin{equation}\label{ffto}\begin{aligned}
&[D^2_{kl}(\mathcal{A}_\omega(\eta) - \lambda_1(\eta))]\varphi_{\omega,1}(\cdot;\eta) + [D_k(\mathcal{A}_\omega(\eta) - \lambda_1(\eta))]D_l\varphi_{\omega,1}(\cdot;\eta)\\
&+ [D_l(\mathcal{A}_\omega(\eta) -\lambda_1(\eta))]D_k \varphi_{\omega,1} (\cdot;\eta) +(\mathcal{A}_\omega(\eta) - \lambda_1(\eta))D^2_{kl}\varphi_{\omega,1}(\cdot;\eta) = 0.
\end{aligned}\end{equation}
Taking scalar product with $\varphi_{\omega,1}(\cdot;\eta)$ in $L^2(\omega)$, we get
\begin{equation*}\begin{aligned}
&\langle D^2_{kl}(\mathcal{A}_\omega(\eta) - \lambda_1(\eta))\varphi_{\omega,1}(\cdot;\eta), \varphi_{\omega,1}(\cdot;\eta)\rangle 
+ \langle [D_k(\mathcal{A}_\omega(\eta) - \lambda_1(\eta))]D_l\varphi_{\omega,1}(\cdot;\eta), \varphi_{\omega,1}(\cdot;\eta)\rangle\\
&+ \langle [D_l(\mathcal{A}_\omega(\eta)- \lambda_1(\eta))]D_k \varphi_{\omega,1}(\cdot;\eta), \varphi_{\omega,1}(\cdot;\eta)\rangle = 0.
\end{aligned}\end{equation*}
Note that $D^2_{kl}\mathcal{A}_\omega(\eta) = 2a^{\omega}_{kl}(y)\quad\forall k,l=1,..,N$.
Evaluating the above relation at $\eta = 0$ and using the information obtained 
in the previous steps, we obtain
\begin{equation*}\begin{aligned}
\frac{1}{2}D^2_{kl}\lambda_1(0) &= \frac{1}{|\omega|}\int_{\omega} a^{\omega}_{kl}(y)dy -
\frac{1}{2|\omega|}\int_{\omega}\{C^{\omega}_k(w_{e_l}(y) - y_l) + C^{\omega}_l(w_{e_k}(y) - y_k)\}dy\\
 &=\ \frac{1}{2|\omega|}\int_{\omega} A_\omega(y)\nabla w_{e_k}(y)\cdot e_l\ dy  + \frac{1}{2|\omega|}\int_{\omega} A_\omega(y)\nabla w_{e_l}(y)\cdot e_k\ dy  \\
 &=\ m_{kl} \quad\forall  k,l=1,.,N,
\end{aligned}\end{equation*}
due to simply using the integral identity \eqref{mkl}. 
They are indeed the homogenized coefficients governed with the Hashin-Shtrikman constructions.
\paragraph{Step 5. Higher order derivatives:}
In general the process can be continued indefinitely to compute all derivatives of
$\lambda_1(\eta)$ and $\varphi_{\omega,1}(\cdot;\eta)$ at $\eta = 0$. In particular, for any $l\in \mathbb{Z}^N_{+}$ with $|l| \geq 2,$ we find 
$D^l \varphi_{\omega,1}(y;0)\in H^1_0(\omega)$ by solving
\begin{equation}\begin{aligned}\label{fe2}
&\mathcal{A}_\omega D^l\varphi_{\omega,1}(y;0) =\ -\sum_{j \neq 0, \ |j| + |k| = |l|} D^{j}(\mathcal{A}_\omega(\eta) - \lambda_1(\eta))|_{\eta=0}\circ D^{k}\varphi_{\omega,1}(y;\eta) |_{\eta = 0} \\
& D^l\varphi_{\omega,1}(y;0) =0 \quad\mbox{on }\partial\omega \quad\mbox{and }\ \int_{\partial\omega}A_{\omega}(y)\nabla_y D^l\varphi_{\omega,1}(y;0)\cdot \nu\ d\sigma = 0
\end{aligned}
\end{equation}
and consequently, $D^q \lambda_1(0)$ with $|q| = |l| +1$.\\
\\
Moreover, it can be shown that all odd order derivatives of $\lambda_1 $ at $\eta=0$ are zero, i.e. 
\begin{equation*}D^{q}\lambda_1(0)=0 \quad\forall q\in \mathbb{Z}^N_{+},\ \ |q| \mbox{ odd}. \end{equation*}
In particular, the third order derivative is zero. However, we are interested in the further next order approximation 
by calculating the fourth order derivatives of $\lambda_1(0)$, i.e.
$D^4_{klmn}\lambda_1(0)$; which is in general a non-positive definite tensor and can be defined as follows:\\
Following \eqref{fe2}, the second order derivative of the eigenvector $D^2_{kl}\varphi_{\omega,1}(\cdot;0)\in H^1_0(\omega)$ solves
\begin{equation*}\begin{aligned}
&\mathcal{A}_\omega D^2_{kl}\varphi_{\omega,1}(y;0)= -( a^{\omega}_{kl}(y) - m_{kl} )\varphi_{\omega,1}(y;0) - iC^{\omega}_k(D_l(\varphi_{\omega,1}(y,0)) - iC^{\omega}_l(D_k\varphi_{\omega,1}(y,0))\mbox{ in }\omega,\\
&D^2_{kl}\varphi_{\omega,1}(y;0) = 0 \quad\mbox{on }\partial\omega \quad\mbox{and}\quad \int_{\partial\omega} A_\omega(y)\nabla_y D^2_{kl}\varphi_{\omega,1}(y;0)\cdot\nu \ d\sigma= 0 .
\end{aligned}\end{equation*}
We call $D^2_{kl}\varphi_{\omega,1}(y;0) = |\omega|^{-1/2}w_{kl}(y)$ and let us define 
\begin{align*}
C_\omega &=\eta_k C^{\omega}_k \quad\mbox{ with }\quad C^{\omega}_k(\varphi) = - a^{\omega}_{kj}(y)\frac{\partial \varphi}{\partial y_j} - \frac{\partial}{\partial y_j}(a^{\omega}_{kj}(y)\varphi), \\
X^{(1)}_\omega &= \eta_k(w_{e_k}(y)-y_k),\quad X^{(2)}_\omega = \eta_k\eta_lw_{kl},\quad  \widetilde{A}_\omega= \eta_k\eta_j a^{\omega}_{kj},\quad \widetilde{M}=\eta_k\eta_j m_{kj},
\end{align*}
where they satisfy
\begin{equation*}\begin{aligned}
-div(A_\omega\nabla X^{(1)}_\omega) &= \eta_k\frac{\partial a^{\omega}_{kl}}{\partial y_l} \mbox{ in }\omega, \\
 X^{(1)}_\omega= 0\mbox{ on }\partial\omega\quad&\mbox{and}\quad \int_{\partial\omega} A_\omega\lb \nabla X^{(1)}_\omega + \eta \rb\cdot\nu\ \ d\sigma  = 0
\end{aligned}\end{equation*} 
and  
\begin{equation*}\begin{aligned}
-div(A_\omega\nabla X^{(2)}_\omega) &= (\widetilde{A}_\omega-\widetilde{M}) - C_{\omega}X^{(1)}_\omega \mbox{ in }\omega,\\
X^{(2)}_\omega = 0\mbox{ on }\partial\omega\quad\mbox{and}&\quad \int_{\partial\omega} A_\omega\nabla X^{(2)}_\omega\cdot\nu\ \ d\sigma  = 0.
\end{aligned}\end{equation*}
Then, by summation, following \cite[Proposition 3.2]{COVB}, it can be shown 
that the following expression defines the fourth order derivative of $\lambda_1(\eta)$ at $\eta=0$:
\begin{equation}\begin{aligned}\label{exp4}
\frac{1}{4!}D^4_{klmn}\lambda_1(0)\eta_k\eta_l\eta_m\eta_n &= -\frac{1}{|\omega|}\int_{\omega} \mathcal{A}_\omega\lb X^{(2)}_\omega - \frac{1}{2}(X^{(1)}_\omega)^2\rb \lb X^{(2)}_\omega - \frac{1}{2}(X^{(1)}_\omega)^2\rb  dy \\
&\leq 0.
\end{aligned}\end{equation}
This tells us that $\lambda_1^{(4)}(\eta)$ at $\eta=0$ is a non-positive definite tensor.
\begin{remark}
In the case of spherical inclusions with two phase materials, i.e.  $\omega=B(0,1)$ and
$A_\omega(y) =\ (\alpha \chi_{B(0,R)}(y) + \beta (1-\chi_{B(0,R)}(y))I$, $y\in B(0,1)$, $R < 1$, for any $l\in \mathbb{Z}^N_{+}$ with $|l|\geq 2$, $D^l\varphi_{\omega,1}(y;0)\in H^1_0(B(0,1))$ solving the Dirichlet boundary value problem \eqref{fe2}  
also satisfies       
\begin{equation}\label{fe3}
 A_\omega(y)\nabla_y D^l\varphi_{\omega,1}(y;0) \cdot \nu = 0 \quad\mbox{on }\partial B(0,1).
\end{equation}
Then the boundary flux/co-normal derivative on the boundary vanishes point-wise.\\
Here we make the following ansatz: for $l =(l_1,l_2,..,l_N)\in \mathbb{Z}^N_{+}$ with $|l| \geq 2,$
$$D^l\varphi_{\omega,1}(y;0) = \sum_{i=1}^N \lb \prod_{j=1}^i y_{l_j} \rb f_i(r) + g(r)\in H^1_0(B(0,1)),$$
which solves the Dirichlet boundary value problem \eqref{fe2} in $B(0,1)$ together with \eqref{fe3}.\\
In particular, for $|l|=2$, we seek
$$ D^2_{kl}\varphi_{\omega,1}(y;0) = y_ky_l f(r) + g(r),$$
with 
$$ f(r) = b + \frac{c}{r^N} + \frac{d}{r^{N+2}}, \quad g(r) = \frac{p}{r^N} + qr^2 + t,$$
where $(b,c,d)$, $(p,q,t)$ are constants and can be found explicitly in terms of $\alpha,\beta,N,\ \theta\ (=R^N)$. \\
For details, see \cite{TV2}.
\end{remark}
\section{The corresponding eigenvalue and eigenvector at $\epsilon_{p,n}$-scale, $y^{p,n}$-translation and Bloch transform}\label{nl10}
\setcounter{equation}{0}
We introduce the operator $\mathcal{A}_\omega^{n}$ motivated from the Hashin-Shtrikman construction
$$\mathcal{A}_\omega^{n} = -\frac{\partial}{\partial x_k}\Big(a^{n}_{kl}(x)\frac{\partial}{\partial x_l}\Big)
\ \mbox{ with  }a_{kl}^{n}(x) =\ a^{\omega}_{kl}\Big(\frac{x-y^{p,n}}{\epsilon_{p,n}}\Big)\ \mbox{ in } \epsilon_{p,n}\omega + y^{p,n} \mbox{ a.e. on } \Omega,$$
where $ meas\big(\Omega \smallsetminus \underset{p\in K}{\cup} (\epsilon_{p,n}\omega + y^{p,n})\big) = 0,$ with 
$\kappa_n = \underset{p\in K}{sup}\hspace{2pt} \epsilon_{p,n}\rightarrow 0$ 
for a finite or countable $K$ and, for each $n$, the sets 
$\epsilon_{p,n}\omega + y^{p,n},\ p\in K$ are disjoint.\\
\\
We obtain the spectral resolution of $\mathcal{A}_\omega^{n}$ for a fixed $n$, in each $\{\epsilon_{p,n}\omega +y^{p,n}\}_{p\in K}$ domain in an analogous manner.  
We introduce the following shifted operator 
\begin{equation*} 
(\mathcal{A}_\omega^{n,p})(\xi)=\ -\Big(\frac{\partial}{\partial x_l} + i\xi_l\Big)\Big( a^{\omega}_{kl}\big(\frac{x-y^{p,n}}{\epsilon_{p,n}}\big)\big(\frac{\partial}{\partial x_k} +i\xi_k\big) \Big),\quad x\in \epsilon_{p,n}\omega +y^{p,n}.
\end{equation*}
By homothecy, for a fixed $n$ and for each $p$, we define the first Bloch eigenvalue $\lambda_1^{n,p}(\xi)$ 
and the corresponding Bloch mode $\varphi_{\omega,1}^{n,p}(\cdot;\xi)$ for the operator $(\mathcal{A}_\omega^{n,p})(\xi)$ for $\xi \in \kappa_n^{-1}\omega^{\prime}$ as follows:
\begin{equation*}
\lambda_1^{n,p}(\xi) := \epsilon_{p,n}^{-2}\lambda_1(\epsilon_{p,n} \xi),\quad 
\varphi^{n,p}_{\omega,1}(x;\xi):= \varphi_{\omega,1}\Big(\frac{x-y^{p,n}}{\epsilon_{p,n}}; \epsilon_{p,n}\xi\Big) \quad \mbox{for } x\in\epsilon_{p,n}\omega + y^{p,n},
\end{equation*}
where $\lambda_1(\eta)$ and $\varphi_{\omega,1}(y;\eta)$ are the eigenelements defined in Section \ref{nl15}. \\
\\
This leads to define the following Bloch transformation in $L^2(\mathbb{R}^N)$ in the following manner:
\begin{proposition}\label{prop1}
\hfill
\begin{enumerate}
\item For $g \in L^2(\mathbb{R}^N)$, for each $n$, the following limit in the space $L^2(\kappa_n^{-1}\omega^{\prime})$ exists:
\begin{equation}\label{btype}
B^{n}_1 g(\xi): = B^{(\epsilon_{p,n},\ y^{p,n})}_1 g(\xi) := \sum_p\int_{(\epsilon_{p,n}\omega + y^{p,n})} g(x)e^{-ix\cdot\xi}\overline{\varphi_{\omega,1}}\Big(\frac{x-y^{p,n}}{\epsilon_{p,n}};\epsilon_{p,n}\xi\Big)dx,
\end{equation}
where for each $n$,  $ meas( \mathbb{R}^N \smallsetminus \underset{p\in K}{\cup} (\epsilon_{p,n}\omega + y^{p,n})) = 0,$ 
with $\kappa_n = \underset{p\in K}{sup}\hspace{2pt} \epsilon_{p,n}\rightarrow 0$ 
for a finite or countable $K$ and the sets $\epsilon_{p,n}\omega + y^{p,n},\ p\in K$ are disjoint.\\
\\
The above definition \eqref{btype} is the corresponding first Bloch transformation governed with
Hashin-Shtrikman micro-structures.
\item We have the following Bessel inequality for elements of $L^2(\mathbb{R}^N)$:
\begin{equation}\label{basel}
\int_{\kappa_n^{-1}\omega^{\prime}} |B^{n}_1 g(\xi)|^2 d\xi \leq \mathcal{O}(1) ||g||^2_{L^2(\mathbb{R}^N)}.
\end{equation}
\item For $g\in H^1(\mathbb{R}^N)$, we have
\begin{equation}\label{rln}
B^{n}_1 \lb\mathcal{A}_\omega^{n}g(\xi)\rb := \sum_p\int_{(\epsilon_{p,n}\omega + y^{p,n})}\lambda_1^{n,p} g(x)e^{-ix\cdot\xi}\overline{\varphi_{\omega,1}}\Big(\frac{x-y^{p,n}}{\epsilon_{p,n}};\epsilon_{p,n}\xi\Big)dx.
\end{equation}
\end{enumerate}
\end{proposition}
\begin{remark}
 For each fixed $n$, the Bloch transform $B_1^n\ =\ B_1^{(\epsilon_{p,n},\ y^{p,n})}$ depends on the choice of 
 Vitali covering. However, in the rest of the paper, we will be using the short notation 
 $B_1^n$ instead of $B_1^{(\epsilon_{p,n},\ y^{p,n})}$.
\end{remark}
\begin{remark}
Spectral decomposition of operators exploiting their group invariance is a classical topic 
in Harmonic analysis. It uses `Group Representation Theory'; while it is 
successfully applied to periodic structures to generate full basis of eigenvectors,
its applicability to Hashin-Shtrikman structures is open.\\
However, for $\xi\in \kappa_n^{-1}\omega^{\prime}$ fixed and for each fixed $n$, 
one has this following spectral decomposition  
$$L^2(\mathbb{R}^N) =\ \bigoplus_{p} L^2(\epsilon_{p,n}\omega + y^{p,n},\epsilon_{p,n}\xi). $$  
Which simply follows from the Theorem \ref{deH} and it is not considered to be the 
desired Bloch spectral decomposition of $L^2(\mathbb{R^N})$.
\end{remark}
\begin{proof}[Proof of Proposition \ref{prop1}]
We notice that
\begin{align*}
&\eta\in \omega^{\prime} \longmapsto ||\varphi_{\omega,1}(\cdot;\eta)||^2_{L^2(\omega)} \mbox{ is analytic. So, } ||\varphi_{\omega,1}(\cdot;\eta)||^2 \mbox{ is bounded on }\omega^{\prime}.\\
\mbox{Thus, }\ \ &\int_{\omega^{\prime}}\int_{\omega}|\varphi_{\omega,1}(y;\eta)|^2 dy d\eta =\ \mathcal{O}(1). 
\end{align*}
Let $g$ be an element of a dense subset of $L^2(\mathbb{R}^N)$. Then, from the definition of $B^{(\epsilon_{p,n},\ y^{p,n})}_1 g(\xi),$ we get 
\begin{align*}
|B^{(\epsilon_{p,n},\ y^{p,n})}_1 g(\xi)| &
\leq\sum_p\int_{\epsilon_{p,n}\omega + y^{p,n}} |g(x)|\Big|{\varphi_{\omega,1}}\Big(\frac{x-y^{p,n}}{\epsilon_{p,n}};\epsilon_{p,n}\xi\Big)\Big|dx \\
&\leq \sum_p\lb \int_{\epsilon_{p,n}\omega + y^{p,n}} |g(x)|^2 dx \rb^{\frac{1}{2}}\lb \int_{\epsilon_{p,n}\omega + y^{p,n}}  \Big|{\varphi_{\omega,1}}\Big(\frac{x-y^{p,n}}{\epsilon_{p,n}};\epsilon_{p,n}\xi\Big)\Big|^2 dx\rb^{\frac{1}{2}} \\
&\leq \mathcal{O}(1)||g||_{L^2(\mathbb{R}^N)}.
\end{align*}
Moreover,
\begin{align*}
\int_{\epsilon_{p,n}^{-1}\omega^{\prime}}|B^{(\epsilon_{p,n},\ y^{p,n})}_1& g(\xi)|^2 d\xi 
=\ \int_{\kappa_n^{-1}\omega^{\prime}}\lb\sum_p\int_{\epsilon_{p,n}\omega + y^{p,n}} g(x)e^{ix\cdot\xi}{\varphi_{\omega,1}}\Big(\frac{x-y^{p,n}}{\epsilon_{p,n}};\epsilon_{p,n}\xi\Big)dx \rb^2 d\xi \\
\leq& \lb\sum_p\int_{\epsilon_{p,n}\omega + y^{p,n}} |g(x)|^2 dx\rb \lb \int_{\kappa_n^{-1}\omega^{\prime}}\sum_p\int_{\epsilon_{p,n}\omega + y^{p,n}}\Big|{\varphi_{\omega,1}}\Big(\frac{x-y^{p,n}}{\epsilon_{p,n}};\epsilon_{p,n}\xi\Big)\Big|^2 dx d\xi \rb\\
\leq&\ \mathcal{O}(1)||g||_{L^2(\Omega)},
\end{align*}
which shows that \eqref{btype} is well defined and belongs to $L^2(\kappa_n^{-1}\omega^{\prime})$
with satisfying \eqref{basel}.\\
The last part \eqref{rln} follows simply from performing integration by parts together with the identity 
\begin{equation*}
\mathcal{A}_\omega(e^{iy\cdot\eta}\varphi_{\omega,1}(y;\eta)) = e^{iy\cdot\eta} \mathcal{A}_\omega(\eta)\varphi_{\omega,1}(y;\eta) = \lambda_1(\eta)e^{iy·\eta}\varphi_{\omega,1}(y;\eta).
\end{equation*}
\hfill\end{proof}
\begin{remark}
One can also define the Bloch transform \eqref{btype} for $H^{-1}(\mathbb{R}^N)$ elements.
Let's consider $F \equiv u^0 + \sum_{j=1}^N \frac{\partial u^j}{\partial x_j}$ $\in H^{-1}(\mathbb{R}^N)$,
then we define $B^{(\epsilon_{p,n},\ y^{p,n})}_1 F(\xi)$ in 
$L^2(\kappa_n^{-1}\omega^{\prime})$ space by using the duality
\begin{equation}\begin{aligned}\label{hm1}
&B^{(\epsilon_{p,n},\ y^{p,n})}_1 F(\xi) := \underset{R \rightarrow \infty}{lim}\Bigg[ \sum_p\int_{(\epsilon_{p,n}\omega + y^{p,n})\cap |x|\leq R} \Big\{u^0(x)e^{-ix\cdot\xi}\overline{\varphi_{\omega,1}^{n,p}}(x;\xi) \\
& + i \sum_{j=1}^N \xi_j u^j(x)\overline{\varphi_{\omega,1}^{n,p}}(x;\xi) \Big\} dx-  \sum_p\int_{(\epsilon_{p,n}\omega + y^{p,n})\cap |x|\leq R}  e^{-ix\cdot\xi}\sum_{j=1}^N u^j(x)\frac{\partial\overline{\varphi_{\omega,1}^{n,p}}}{\partial x_j}(x;\xi)dx\Bigg]. 
\end{aligned}
\end{equation}
The definition is independent of the representation used for $F$ and is consistent with 
the fact that $F\in L^2(\mathbb{R}^N)$.
\end{remark}
\begin{remark}
By  homothecy, because we have $\lambda_1^{n,p}(\xi) = \epsilon_{p,n}^{-2} \lambda_1(\epsilon_{p,n}\xi)$, using the Taylor expansion of $\lambda_1(\eta)$ at $\eta=0$,
we get
\begin{equation*}
\lambda_1^{n,p}(\xi) = \frac{1}{2!}\frac{\partial^2\lambda_1}{\partial\eta_k\partial\eta_l}(0)\xi_k\xi_l + \epsilon_{p,n}^2\frac{1}{4!}\frac{\partial^4\lambda_1}{\partial\eta_k\partial\eta_l\partial\eta_m\partial\eta_n}(0)\xi_k\xi_l\xi_m\xi_n + \mathcal{O}(\epsilon_{p,n}^2).
\end{equation*}
Thus, from \eqref{rln}, we write
\begin{equation*}\begin{aligned}
B^{n}_1\lb\mathcal{A}_\omega^{n}g(\xi)\rb &= \sum_p\int_{\epsilon_{p,n}\omega + y^{p,n}}\lambda_1^{n,p} g(x)e^{-ix\cdot\xi}\overline{\varphi_{\omega,1}}\Big(\frac{x-y^{p,n}}{\epsilon_{p,n}};\epsilon_{p,n}\xi\Big)dx\\
=&\ \frac{1}{2!}\frac{\partial^2\lambda_1}{\partial\eta_k\partial\eta_l}(0)\xi_k\xi_l\hspace{2pt} B^{n}_1g(\xi) + \frac{1}{4!}\frac{\partial^4\lambda_1}{\partial\eta_k\partial\eta_l\partial\eta_m\partial\eta_n}(0)\xi_k\xi_l\xi_m\xi_n\cdot \\
 &\cdot\sum_p\int_{\epsilon_{p,n}\omega + y^{p,n}}\epsilon_{p,n}^2 g(x)e^{-ix\cdot\xi}\overline{\varphi_{\omega,1}}\Big(\frac{x-y^{p,n}}{\epsilon_{p,n}};\epsilon_{p,n}\xi\Big)dx + o(\kappa_n^2).
\end{aligned}\end{equation*}
Here in the above expression, the first term or the second order approximation provides 
the homogenized coefficients. The second term provides the fourth order approximation
related to the dispersion tensor or  \textit{Burnett coefficient}
in  Hashin-Shtrikman structures. For more details, see \cite{TV2}. 
\end{remark}

\subsection{First Bloch transform $B_1^{n}$ tends to Fourier transform}
 
Here we establish the limiting association of 
the first Bloch transform $B_1^{n}g(\xi)$ of $g\in L^2(\mathbb{R}^N)$ 
associated with the Hashin-Shtrikman structures
and the Fourier transform $\widehat{g}(\xi)$ which is 
similarly associated with a homogeneous medium. 
\begin{theorem}\label{t7}
\noindent
\begin{enumerate}
 \item  If $g_{n} \rightharpoonup g$ in $L^2(\mathbb{R}^N)$ weak, then
$\chi_{\kappa_n^{-1}\omega^{\prime}}(\xi)B_1^{n}g^{n}(\xi) \rightharpoonup \widehat{g}(\xi)$
in $L^2(\mathbb{R}^N)$ weak provided there is a fixed compact  $R$ 
such that support of $ g^n\subseteq R\quad\forall n.$ 

\item  If $g_{n} \rightarrow g$ in $L^2(\mathbb{R}^N)$ strong, then for the subsequence $\epsilon_{p,n}$
$\chi_{\kappa_n^{-1}\omega^{\prime}}(\xi)B_1^{n}g^n \rightarrow \widehat{g}(\xi)$
in $L^2_{loc}(\mathbb{R}^N).$
\end{enumerate}
\end{theorem}
\begin{proof}
1. Let us consider a Vitali covering for $R$ to write as 
$R \approx \underset{p\in K}{\cup} (\epsilon_{p,n}\omega + y^{p,n})$, 
with $\kappa_n = \underset{p\in K}{sup}\hspace{2pt} \epsilon_{p,n}\rightarrow 0$ 
for a finite or countable $K$ and, for each $n$, the sets $\epsilon_{p,n}\omega + y^{p,n},\ p\in K$ are disjoint. \\
Then, from the definition of $B^{(\epsilon_{p,n},\ y^{p,n})}_1 g(\xi)$  for $\xi$ in $\kappa_n^{-1}\omega^{\prime}$, 
\begin{align*}
B_1^{n}g^n(\xi) =& \sum_p \int_{\epsilon_{p,n}\omega + y^{p,n}} g^n(x)e^{-ix\cdot\xi}\overline{\varphi_{\omega,1}}\Big(\frac{x-y^{p,n}}{\epsilon_{p,n}};0\Big)dx\\
            &+ \sum_p \int_{\epsilon_{p,n}\omega + y^{p,n}} g^{n}(x)e^{-ix\cdot\xi}\lb\overline{\varphi_{\omega,1}}\Big(\frac{x-y^{p,n}}{\epsilon_{p,n}};\epsilon_{p,n}\xi\Big)- \overline{\varphi_{\omega,1}}\Big(\frac{x-y^{p,n}}{\epsilon_{p,n}}; 0\Big)\rb dx. 
\end{align*}
Since $\varphi_{\omega,1}(y;0) = |\omega|^{-1/2}$, the first term of the above identity  is nothing else that the Fourier transform of
$g^n$ and so it converges to $g$ in $L^2(\mathbb{R})$ weak.
For the second term, we apply the Cauchy-Schwarz inequality to bound it from above
as follows:\begin{align*} \mbox{ Second term } \leq&\ \sum_p \|g^n\|_{L^2(\epsilon_{p,n}\omega + y^{p,n})} \Big\|\varphi_{\omega,1}\Big(\frac{x-y^{p,n}}{\epsilon_{p,n}};\epsilon_{p,n}\xi\Big)- \varphi_{\omega,1}\Big(\frac{x-y^{p,n}}{\epsilon_{p,n}}; 0\Big)\Big\|_{L^2(\epsilon_{p,n}\omega +y^{p,n})}\\
                                      \leq&\ \sum_p \|g^n\|_{L^2(\epsilon_{p,n}\omega + y^{p,n})}\cdot C\epsilon_{p,n}|\xi| \\
                                         \leq&\ \kappa_n |\xi|\|g^n\|_{L^2(\mathbb{R}^N)}.
\end{align*}                                         
The second inequality follows simply due to the analyticity of $\varphi_{\omega,1}(\cdot;\eta) \in L^2 (\omega)$ near $\eta=0$.
And thus, it converges to zero in $L^{\infty}_{loc}(\mathbb{R}^N_{\xi})$.
This completes the proof of $1$.\\
\\
2. Let us first consider the case 
$g^n = g$, where $g \in L^2(\mathbb{R}^N)$ is with compact support. 
Following the
proof of $1$, we have $B_1^{n}g \rightarrow g$ in $L^2_{loc}(\mathbb{R}^N)$ strong.\\
\\
In general, we introduce the operator $B_1^{n} : L^2(\mathbb{R}^N) \to L^2(\mathbb{R}^N)$
and following the Basel inequality \eqref{basel}, 
it shows that $\|B_1^{n}\|_{L^2\mapsto L^2}\leq \mathcal{O}(1).$
Now, we complete the rest of the proof $2$ by density arguments.
Take $ g\in L^2(\mathbb{R}^N)$  arbitrary, 
we approximate it by $h\in L^2(\mathbb{R}^N)$ with compact support. 
Then, the desired result follows via triangle inequality applied to the relation
\begin{equation*}
B_1^{n}g - \widehat{g} = B_1^{n}(g - h) + (B_1^{n} h - \widehat{h}) + (\widehat{h} - \widehat{g} ).
\end{equation*}
Finally, if $g^n \rightarrow g$ in $L^2(\mathbb{R}^N)$ strong, then 
\begin{equation*}
B_1^{n}g^n - \widehat{g} = B_1^{n}(g^n- g) + (B_1^{n} g - \widehat{g})
\end{equation*}
shows that $B_1^{n}g_{n} \rightarrow g$ in $L^2_{loc}(\mathbb{R}^N)$. This completes the proof of $2$.
\hfill\end{proof}
\begin{remark}
The factor $\chi_{\kappa_n^{-1}\omega^{\prime}}(\xi)$  in the above result 
was merely used to extend the relevant
functions by zero outside their domain of definition. 
It did not play any part in the proof
because we are interested in local convergence.
\end{remark}

\section{Homogenization result}\label{nl16}
\setcounter{equation}{0}
The purpose of this section is to provide a proof of the main 
result of homogenization stated in preliminary section. 
It will be based on the tools whatever we have derived 
in the previous sections.
\begin{theorem}
Let us consider $\Omega$ be an open set in $\mathbb{R}^N$. 
We introduce the operator $\mathcal{A}_\omega^{n}$ governed with the Hashin-Shtrikman construction (Example \ref{hsm}):
$$\mathcal{A}_\omega^{n} = -\frac{\partial}{\partial x_k}\Big(a^{n}_{kl}(x)\frac{\partial}{\partial x_l}\Big)
\ \mbox{ with  }a_{kl}^{n}(x) =\ a^{\omega}_{kl}\Big(\frac{x-y^{p,n}}{\epsilon_{p,n}}\Big)\ \mbox{ in } \epsilon_{p,n}\omega + y^{p,n} \mbox{ a.e. on } \Omega,$$
where $ meas\big(\Omega \smallsetminus \underset{p\in K}{\cup} (\epsilon_{p,n}\omega + y^{p,n})\big) = 0,$ with 
$\kappa_n = \underset{p\in K}{sup}\hspace{2pt} \epsilon_{p,n}\rightarrow 0$ 
for a finite or countable $K$ and, for each $n,$ the sets 
$\epsilon_{p,n}\omega + y^{p,n},\ p\in K$ are disjoint.
And $A_\omega$ is equivalent to $M$ $\in L^{+}(\mathbb{R}^N;\mathbb{R}^N)$,
then extending $A_\omega$ by $A_\omega(x) = M$ for $x \in \mathbb{R}^N\smallsetminus \omega$ 
for all $\lambda \in \mathbb{R}^N$, there exists $w_{\lambda}\in  H^1_{loc}(\mathbb{R}^N)$ satisfying
\begin{equation*}
- div (A_\omega\nabla w_{\lambda}(y)) = 0 \quad\mbox{in }\mathbb{R}^N,\quad  w_{\lambda}(y) = (\lambda, y) \quad\mbox{in }\mathbb{R}^N \smallsetminus \omega.
\end{equation*}
Let $f \in L^2(\Omega)$ and $u^{n} \in H^1_0(\Omega)$ be the
unique solution of the boundary value problem
\begin{equation*} \mathcal{A}_\omega^{n}u^{n} = f \quad\mbox{in }\Omega.\end{equation*}
Then there exists $u\in H^1_{0}(\Omega)$ such that the sequence $u^{n}$ converges to $u$ in $H^1_{0}(\Omega)$ weak
with the following convergence of flux
\begin{equation*}
\sigma^{n}_\omega= A^{n}_\omega\nabla u^{n} \rightharpoonup M\nabla u =\sigma_\omega \quad\mbox{in } L^2(\Omega)^N\mbox{ weak. }
\end{equation*}
In particular, the limit $u$ satisfies the homogenized equation:
\begin{equation*}-\frac{\partial}{\partial x_k}\big(m_{kl}\frac{\partial}{\partial x_l}u\big) = f \quad\mbox{in }\Omega.\end{equation*}
\end{theorem}

\begin{proof}
We start with the cut-off function technique to localize the equation.
\paragraph{Step 1. Localization:}
\noindent
Let $v \in D(\Omega)$ be arbitrary. Then the localization $vu^{n}$ satisfies
\begin{equation}\label{local} \mathcal{A}_\omega^{n}(vu^{n}) = vf + g^{n}_\omega + h^{n}_\omega\quad\mbox{in }\mathbb{R}^N, \end{equation}
where\begin{equation*}
 g^{n}_\omega =\ -2a^{n}_{kl}\frac{\partial u^{n}}{\partial x_l}\frac{\partial v}{\partial x_k} - a^{n}_{kl}\frac{\partial^2v}{\partial x_k\partial x_l}u^{n},\quad\quad
h^{n}_\omega  =\ -\frac{\partial a^{n}_{kl}}{\partial x_k}\frac{\partial v}{\partial x_l}u^{n}.
\end{equation*}
$g^{n}_\omega$ and $h^{n}_\omega$
correspond to terms containing zero and first order derivatives on $a^{n}_{kl}$, respectively.
\paragraph{Step 2. Limit of LHS of \eqref{local}:}
\noindent
We consider the first Bloch transform $B_1^{n}$ given in \eqref{btype} on L.H.S. of \eqref{local}
to get $\lambda_1^{n,p}(\xi)B_1^{n}(vu^{n})$. 
Since $v$ has compact support,  $vu^{n} \rightarrow vu$ in $L^2(\mathbb{R}^N)$ strong, then by using Theorem \ref{t7}, we get 
\begin{equation*}
\chi_{\kappa_n^{-1}\omega^{\prime}}(\xi)\lambda_1^{n,p}(\xi)B_1^{n}(vu^{n}) \rightarrow \frac{1}{2}D^2_{kl}\lambda_1(0)\xi_k\xi_l\widehat{vu}(\xi)\quad\mbox{in }L^2_{loc}(\mathbb{R}^N)\mbox{ strong.}
\end{equation*}
\paragraph{Step 3. Limit of $B_1^{n}g^{n}_\omega$:}
\noindent
Since $\sigma^{n}_\omega$ is bounded in $L^2(\Omega)^N$, 
there exists a convergent subsequence with limit $\sigma_\omega \in L^2(\Omega)^N$ and
we extend it by zero outside $\Omega$.
Thus,
we have\begin{equation*}
g^{n}_\omega \rightharpoonup g_\omega =\ -2\sigma^{\omega}_k\frac{\partial v}{\partial x_k} - M_{\omega}({a^{\omega}_{kl}})\frac{\partial^2v}{\partial x_k\partial x_l}u \quad\mbox{in }L^2_{loc}(\mathbb{R}^N)\mbox{ weak},
\end{equation*}
where $M_{\omega}({a^{\omega}_{kl}})$ is the $L^{\infty}$-weak* limit of $\{a^{n}_{kl}\}$
satisfying $M_{\omega}({a^{\omega}_{kl}}) =\frac{1}{|\omega|}\int_{\omega} a^{\omega}_{kl}(y) dy$ 
which follows from  Lemma \ref{A2}.\\
\\
Thus, by applying Theorem \ref{t7}, we have
\begin{equation*}
\chi_{\kappa_n^{-1}\omega^{\prime}}(\xi)B_1^{n}g^{n}_\omega(\xi) \rightharpoonup \widehat{g}_\omega(\xi)
\quad\mbox{in }L^2(\mathbb{R}^N)\mbox{ weak}. 
\end{equation*}
Due to integration by parts, we get
\begin{equation*}
\widehat{g}_\omega(\xi) = \frac{1}{|\omega|^{1/2}}\int_{\mathbb{R}^N} \Big[-2\sigma^{\omega}_k\frac{\partial v}{\partial x_k} + M_{\omega}(a^{\omega}_{kl})\frac{\partial v}{\partial x_l}\frac{\partial u}{\partial x_k} - (i\xi_k)M_{\omega}(a^{\omega}_{kl})\frac{\partial v}{\partial x_l}u\Big] e^{-ix\cdot\xi} dx.
\end{equation*}
\paragraph{Step 4. Limit of $B_1^{n}h^{n}_\omega(\xi)$:}
\noindent
Here we see that  $h^{n}_\omega$ is uniformly supported in a fixed compact set (say $R$) and
bounded in $H^{-1}(\mathbb{R}^N)$ but not in $L^2(\mathbb{R}^N)$. 
Then, in order to calculate $B_1^{n}h^{n}_\omega(\xi)$, we use the idea of 
decomposition what we have mentioned in \eqref{hm1}.
Let us consider a Vitali covering for $R$ to write as 
$R \approx \underset{p\in K}{\cup} (\epsilon_{p,n}\omega + y^{p,n})$ 
with $\kappa_n = \underset{p\in K}{sup}\hspace{2pt} \epsilon_{p,n}\rightarrow 0$ 
for a finite or countable $K$ and, for each $n$, the sets $\epsilon_{p,n}\omega + y^{p,n},\ p\in K$ are disjoint.
We have\begin{equation}\begin{aligned}\label{abv0}
B_1^{n}h^{n}_\omega(\xi) = &\sum_p\int_{\epsilon_{p,n}\omega + y^{p,n}} h^{n}_\omega(x)e^{-ix\cdot\xi}\overline{\varphi_{\omega,1}}\Big(\frac{x-y^{p,n}}{\epsilon_{p,n}};0\Big)dx \\
&+ \sum_p\int_{\epsilon_{p,n}\omega + y^{p,n}} h^{n}_\omega(x)e^{-ix\cdot\xi}\lb\overline{\varphi_{\omega,1}}\Big(\frac{x-y^{p,n}}{\epsilon_{p,n}};\epsilon_{p,n}\xi\Big)- \overline{\varphi_{\omega,1}}\Big(\frac{x-y^{p,n}}{\epsilon_{p,n}}; 0\Big)\rb dx.  
\end{aligned}\end{equation}
We start with the second term, by following the Taylor expansion around zero  of $\varphi_{\omega,1}(y;\cdot)$, we get
\begin{equation*}
-\sum_p\int_{\epsilon_{p,n}\omega + y^{p,n}}\frac{\partial a^{n}_{kl}}{\partial x_k}\frac{\partial v}{\partial x_l}u^{n}  
e^{-ix\cdot\xi}\lb \epsilon_{p,n}\frac{\partial\overline{\varphi_{\omega,1}}}{\partial \eta_j}\Big(\frac{x-y^{p,n}}{\epsilon_{p,n}};0\Big)\xi_j + \mathcal{O}(\epsilon_{p,n}^2\xi^2)\rb dx, 
\end{equation*}
which, via integrating by parts, becomes 
\begin{equation}\label{abv1}
\xi_j \sum_p\int_{\epsilon_{p,n}\omega + y^{p,n}} a^{n}_{kl}\frac{\partial v}{\partial x_l}u^{n}e^{-ix\cdot\xi}
\frac{\partial^2\overline{\varphi_{\omega,1}}}{\partial \eta_j\partial y_k}\Big(\frac{x-y^{p,n}}{\epsilon_{p,n}};0\Big) dx  + \mathcal{O}(\epsilon_{p,n}\xi).
\end{equation}
(the boundary term vanishes as $\frac{\partial\varphi_{\omega,1}}{\partial\eta_j}(\cdot;\eta)|_{\eta=0}\in H^1_0(\omega)$).\\

We claim that the above integral term converges in $L^2_{loc}(\mathbb{R}^N)$ to
\begin{equation}\label{abv2}
L_\omega\cdot\xi_j \int_{\mathbb{R}^N} \frac{\partial v}{\partial x_l}u e^{-ix\cdot\xi} dx,
\end{equation}
where 
\begin{align*}
L_\omega = &\ L^{\infty}\mbox{-weak* limit }\lb a^{\omega}_{kl}\Big(\frac{x-y^{p,n}}{\epsilon_{p,n}}\Big)\frac{\partial^2\overline{\varphi_{\omega,1}}}{\partial \eta_j\partial y_k}\Big(\frac{x-y^{p,n}}{\epsilon_{p,n}};0\Big)\rb\\
  =&\ \frac{1}{|\omega|}\int_{\omega}a^{\omega}_{kl}(y)\frac{\partial^2\overline{\varphi_{\omega,1}}}{\partial \eta_j\partial y_k}(y;0)dy \quad\mbox{(it follows from Lemma }\ref{A2}).
\end{align*}
Let $I^{n}(x) = a^{n}_{kl}\frac{\partial v}{\partial x_l}u^{n}e^{-ix\cdot\xi}
\frac{\partial^2\overline{\varphi_{\omega,1}}}{\partial \eta_j\partial x_k}(\frac{x-y^{p,n}}{\epsilon_{p,n}};0)\in L^1(\mathbb{R}^N_x)$,
then the above integrand term \eqref{abv1} is $\widehat{I^{n}}(\xi)\in L^{\infty}_{loc}(\mathbb{R}^N_{\xi})$
and consequently $\widehat{I^{n}}(\xi)\rightarrow \widehat{I}(\xi)$ $\forall \xi\in\mathbb{R}^N$,
where $I(x)$ is given by \eqref{abv2}. Thus $\xi_j\widehat{I^{n}}(\xi)\rightarrow \xi_j\widehat{I}(\xi)$ 
in $L^2_{loc}(\mathbb{R}^N)$ strongly.\\
\\
Now, let us consider the first term of the right hand side of \eqref{abv0}. After doing integration by parts, 
one has
\begin{equation*}
\frac{1}{|\omega|^{\frac{1}{2}}}\int_{\mathbb{R}^N} a^{n}_{kl} 
\Big[\frac{\partial^2v}{\partial x_k\partial x_l}u^{n} + \frac{\partial v}{\partial x_l}\frac{\partial u^{n}}{\partial x_k} -i\xi_k\frac{\partial v}{\partial x_l}u^{n}\Big]e^{-ix\cdot\xi} dx . 
\end{equation*}
By the similar way as we just have done, the limit of the above term would be
\begin{equation*}
\frac{1}{|\omega|^{\frac{1}{2}}}\int_{\mathbb{R}^N}\Big[M_{\omega}(a^{\omega}_{kl}) 
\frac{\partial^2v}{\partial x_k\partial x_l}u + \sigma^{\omega}_l\frac{\partial v}{\partial x_l} - (i\xi_k) M_{\omega}(a^{\omega}_{kl})\frac{\partial v}{\partial x_l}u\Big]e^{-ix\cdot\xi} dx  
\end{equation*}
and again performing the integration by parts in the first term, we see the above term is equal to
\begin{equation}\label{1st}
\frac{1}{|\omega|^{\frac{1}{2}}}\int_{\mathbb{R}^N}\Big( \sigma^{\omega}_l\frac{\partial v}{\partial x_l} - M_{\omega}(a^{\omega}_{kl})\frac{\partial v}{\partial x_l}\frac{\partial u}{\partial x_k}\Big) e^{-ix\cdot\xi}dx. 
\end{equation}
Now combining \eqref{1st} and \eqref{abv2} and using the fact 
$\frac{\partial\overline{\varphi_{\omega,1}}}{\partial \eta_j}(y;0) = -i|\omega|^{-1/2}(w_{e_j}(y)-y_j)$,
we see that $\chi_{\kappa_n^{-1}\omega^{\prime}}B_1^{n}h^{n}_\omega(\xi)$ converges 
strongly in $L^2_{loc}(\mathbb{R}^N)$ to
\begin{equation*}\begin{aligned}
-|\omega|^{-1/2}M_{\omega}\Big( a^{\omega}_{kl}\frac{\partial (w_{e_j}(y)-y_j)}{\partial y_k}\Big)(i\xi_j)
&\int_{\mathbb{R}^N} \frac{\partial v}{\partial x_l}u e^{-ix\cdot\xi}dx \\
+|\omega|^{-1/2}\int_{\mathbb{R}^N}&\Big( \sigma^{\omega}_l \frac{\partial v}{\partial x_l} 
- M_{\omega}(a^{\omega}_{kl})\frac{\partial v}{\partial x_l}\frac{\partial u}{\partial x_k}\Big) e^{-ix\cdot\xi} dx.
\end{aligned}\end{equation*}
\paragraph{Step 5. Limit of \eqref{local}:}
\noindent
By taking the Bloch transformation \eqref{btype} of the equation \eqref{local} and then, by passing to the limit onto it,
we get
\begin{equation}\label{lch}
\widehat{\mathcal{M}(vu)}(\xi) = \widehat{vf}(\xi) - |\omega|^{-1/2}\int_{\mathbb{R}^N} \sigma^{\omega}_k\frac{\partial v}{\partial x_k}e^{-ix\cdot\xi} dx 
 -(i\xi_k)|\omega|^{-1/2}m_{kl}\int_{\mathbb{R}^N} u\frac{\partial v}{\partial x_l} e^{-ix\cdot\xi} dx,
\end{equation}
where $\mathcal{M} \equiv -\frac{\partial}{\partial x_k}(m_{kl}\frac{\partial}{\partial x_l})$ the homogenized operator.\\
The above equation is considered as the localized homogenized equation in the Fourier space. 
The conclusion of the Theorem will follow as a consequence of this equation.
\paragraph{Step 6. Fourier space ($\xi)$ to physical space $(x)$: }
We take the inverse Fourier transform of the localized homogenized equation \eqref{lch} to go back to physical space
\begin{equation*}
\mathcal{M}(vu) = vf - \sigma^{\omega}_k\frac{\partial v}{\partial x_k} - m_{kl}\frac{\partial}{\partial x_k}\Big(\frac{\partial v}{\partial x_l}u\Big) \quad\mbox{in }\mathbb{R}^N.
\end{equation*}
The left hand side part $\mathcal{M}(vu)$ can be calculated directly from the definition of the operator $\mathcal{M}$:
\begin{equation*}
\mathcal{M}(vu) = -m_{kl} \frac{\partial^2 v}{\partial x_k\partial x_l} u - 2 m_{kl}\frac{\partial v}{\partial x_k}\frac{\partial u}{\partial x_l} + v \mathcal{M}(u) \quad\mbox{in }\mathbb{R}^N. 
\end{equation*}
Thus, by equating with the right hand side part, it follows 
\begin{equation}\label{re}
v(\mathcal{M}(u) -f) = \lb m_{kl}\frac{\partial u}{\partial x_l} - \sigma_k \rb \frac{\partial v}{\partial x_k} \quad\mbox{in }\mathbb{R}^N \quad \forall v\in D(\Omega).
\end{equation}
Let us choose $v(x) = v_0(x)e^{inx\cdot\zeta}$, where $\zeta$
is a unit vector in $\mathbb{R}^N$ and $v_0 \in D(\Omega)$ is fixed.
Then by letting $n\rightarrow \infty$ in the resulting relation \eqref{re}
and varying the unit vector $\zeta$, we can easily deduce
successively that 
$$\sigma^{\omega}_k = m_{kl}\frac{\partial u}{\partial x_l}\quad\mbox{ in }\Omega\quad\mbox{ and  }\quad -\frac{\partial}{\partial x_k}\Big(m_{kl}\frac{\partial}{\partial x_l}u\Big) = f \quad\mbox{ in }\Omega.$$
This completes our proof of the Main Theorem of Homogenization.
\hfill\end{proof}
\section{Bloch spectral representation of a class of non-periodic simple laminates in two-phase medium}
\setcounter{equation}{0}
In this section, we present Bloch spectral representation of one subclass of simple laminates which are non periodic structures.
It is governed with both non-uniform scales and transformations in one direction and maintains uniformity with respect to scales and translation in 
other directions. As a particular case, it includes the periodic laminates also.
\paragraph{Laminated micro-structure:}
The laminated micro-structures are defined as where the geometry of the problem varies 
only in a single direction, that means the sequence of matrices $A^{n}$ depends on a single
space variable,  $A^{n}(x) = A^{n}(x\cdot e)$ (where $e$ is some standard basis vector in $\mathbb{R}^N$)
and the homogenized composite is called laminates. If the component phases are stacked in slices orthogonal 
to the $e$ direction, in that case it is a generalization of the one-dimensional settings. 
In particular, the $H$-convergence can be reduced to the usual weak convergence 
of some combinations of entries of the matrix $A^{n}$. In effect, this yields another type of 
explicit formula for the homogenized matrix as in the one-dimensional case. \\
Let us consider this following result, $A^{n} \in \mathcal{M}(\alpha,\beta,\Omega)$ satisfying the assumption
\begin{center}$A^{n}(x) = a^{n}(x\cdot e_i)I$ a.e. $x\in\Omega$.\end{center} 
Then, $A^{n}$ $H$-converges to an homogenized matrix $A^{*}$ if and only if the following convergences hold in  $L^{\infty}(\Omega)$-weak*  (see \cite{A}):
$$\frac{1}{A^{n}_{ii}} \rightharpoonup \frac{1}{A^{*}_{ii}}\ \mbox{ and } A^{n}_{jj}\rightharpoonup A^{*}_{jj} \ \mbox{ for }1\leq i \leq N \mbox{ and }j\neq i\ \mbox{ in  $L^{\infty}(\Omega)$-weak*. }$$
\begin{center}So, $A^{*} = diag(\overline{a}(x\cdot e_i),..,\underline{a}(x\cdot e_i),..,\overline{a}(x\cdot e_i)),\quad$ $\underline{a}(x\cdot e_i)$ comes at the $i$-th diagonal entry.\end{center}
Here $\underline{a}(x\cdot e)$ and $\overline{a}(x\cdot e)$ call the harmonic mean and 
arithmetic mean of $a^{n}(x\cdot e)$.
\paragraph{A subclass of non-periodic laminates in two-phase medium:}
Here, let us consider $S=[-1,1]^N$ and for some fixed $i\in \{1,\ldots,N\}$ define $A_S\in \mathcal{M}(\alpha,\beta;S)$ as
\begin{equation}\begin{aligned}\label{nl6}
A_S(y) = a_S(y\cdot e_i)I &=
\left\{\begin{array}{ll} \alpha I \ \mbox{ whenever } y\cdot e_i \in [-a_i,a_i]\subset [-1,1],\\[1ex] 
                      \beta I \ \mbox{ elsewhere.}
                     \end{array} \right.
\end{aligned}\end{equation}
The volume fraction $\theta$ is $\theta=\ a_i$.\\
\\
Now we seek the equivalence of $a_S(y\cdot e_i)$ to some $m$ in $e_i$ direction.  
Then, extending $a_S(y\cdot e_i)$ by $m$ for $y\cdot e_i \in \mathbb{R}\smallsetminus [-1,1]$ where $m$ is a positive constant, 
if there exists $w_{e_i}\in  H^{1}_{loc}(\mathbb{R})$ satisfies (for some fixed $i\in \{1,..,N\}$):
\begin{equation}\label{nl1}
- \frac{d}{dy_i}\Big(a_S(y_i)\frac{d}{dy_i}w_{e_i}(y_i)\Big) =\ 0 \quad\mbox{in }\mathbb{R},\quad  w_{e_i}(y_i) = y_i \quad\mbox{in }\mathbb{R} \smallsetminus [-1,1].
\end{equation}
Then $a_S(y\cdot e_i)$ is said to be \textit{equivalent} to $m$.\\
\\
From the Example \ref{si}, restricting it in one-dimension ($N=1$) case, we establish the existence of $w_{e_i}\in H^{1}_{loc}(\mathbb{R})$ satisfying \eqref{nl1}
and $a_S(y\cdot e_i)$ is equivalent to 
$$m = \frac{\alpha\beta}{\theta\beta +(1-\theta)\alpha}=\underline{a}\mbox{ (say), }\mbox{ (the harmonic mean of $\alpha$ and $\beta$ with the proportion $\theta$)}. $$ 
In the other directions $e_j$ ( $j=1,..,(i-1),(i+1),..,N$ ) we will be using the periodic arrays to define the micro-structures as follows:
one uses a sequence of Vitali coverings $(\epsilon_{p,n},y^{p,n})$ of $\Omega$ by reduced copies of $S$ in $e_i$
direction and uses $\epsilon_{p,n}S$-periodicity in the directions orthogonal to $e_i$ on the cell centered at  $y^{p,n}$ to get
\begin{equation*}  meas\ (\Omega \smallsetminus\ \underset{z\in \mathbb{Z}^{N-1}}{\cup}\underset{p\in K}{\cup}(\epsilon_{p,n}(S+z) + y^{p,n})) = 0, \mbox{ with } \kappa_n = \underset{p\in K}{sup}\hspace{2pt} \epsilon_{p,n}\rightarrow 0,\end{equation*}
for a finite or countable $K$. These define the micro-structures in $A^{n}$ as
\begin{equation}\label{nl2} A^{n}_S(x) = a_S\Big(\frac{x - y^{p,n}}{\epsilon_{p,n}}\cdot e_i\Big)I\ \mbox{ in } \epsilon_{p,n}(S+z) + y^{p,n}\ \mbox{ a.e. in }\Omega, \quad p\in K,\ z\in\mathbb{Z}^{N-1},\end{equation}
which makes sense since for each $n$ the sets $\epsilon_{p,n}(S+z) + y^{p,n},\ p\in K,z\in\mathbb{Z}^{N-1}$ are disjoint.
The above construction \eqref{nl2} represents one subclass of \textit{non-periodic laminate} micro-structures in two-phase medium.\\
\\
Consequently, one has the following $H$-convergence of the entire sequence
\begin{equation}\label{nl5} A^{n}_S \xrightarrow{H \mbox{ converges }} A^{*}_S=\ diag(\overline{a_S},..,\underline{a_S},..,\overline{a_S}),\,\,\mbox{ $\underline{a_S}$ comes at only $i$-th diagonal entry. } \end{equation}
Due to the periodicity in $e_j$ ($j\neq i$) directions, the homogenized conductivity $(A^{*}_S)_{jj}$ ($j\neq i$) is defined by its entries:
\begin{equation}\label{nl3}
(A^{*}_S)_{jj} = \frac{1}{|S|}\int_S a_S(y\cdot e_i)(\nabla\chi_j + e_j)\cdot (\nabla\chi_j +e_j) dy, \quad j=1,..,(i-1),(i+1),..,N,
\end{equation}
where for each $j\neq i$,  $\chi_j\in H^1(S)$ solves the following cell problem 
\begin{equation}\label{nl4}
-div\ (a_S(y\cdot e_i)(\nabla\chi_j(y) + e_j)) = 0 \ \mbox{ in }S,\ \ y \mapsto\chi_j(y)\ \mbox{ is $1$-periodic in each $e_j\ (j\neq i$) direction.}
\end{equation}
As we see, $\chi_j(y) = 0$ ($j\neq i$) uniquely solves the above equation to give
$$ (A^{*}_S)_{jj} = \theta\alpha + (1-\theta)\beta= \overline{a}\mbox{ (say)},\mbox{ ( the arithmetic mean of $\alpha$ and $\beta$ with the proportion $\theta$ )}. $$ 
So, the limit in \eqref{nl5} is well understood now. 
\begin{remark}
As we see, the micro-structures governed by \eqref{nl2} are periodic in $(N-1)$ directions and in one direction it includes 
one-dimensional Hashin-Shtrikman construction.
\end{remark}
\paragraph{Bloch spectral analysis:}
We take $A_S(y)= a_S(y\cdot e_i)I\in \mathcal{M}(\alpha,\beta;S)$ defined in \eqref{nl6} to consider
the following spectral problem parameterized by $\eta \in \mathbb{R}^N$: 
Find $\mu := \mu(\eta) \in \mathbb{C}$ and $\varphi_S := \varphi_S(y; \eta)$ (not identically zero) such that
\begin{equation}\begin{aligned}\label{nl7}
-\Big(\frac{\partial}{\partial y_k} + i\eta_k\Big)\Big[a_S(y\cdot e_i)\Big(\frac{\partial}{\partial y_k} + i\eta_k\Big)\Big]\varphi_S(y;\eta) &= \mu(\eta)\varphi_S(y;\eta) \mbox{ in }S, \\
\varphi_S(y;\eta)\mbox{ is constant on }y\cdot e_i &= \pm 1,  \\
\varphi_S(y;\eta)\mbox{ is $1$-periodic in each $e_j$ }&( j\neq i )\mbox{ direction,}\\
\int_{y\cdot e_i =\pm 1} a_S(y\cdot e_i)\Big(\frac{\partial}{\partial y_k} + i\eta_k\Big)\varphi_S(y;\eta)\nu_k\ d\sigma &= 0, 
\end{aligned}\end{equation}
(where $\nu= \pm e_i$ is the outer normal unit vector on the boundary $\{y\cdot e_i =\pm 1\}$  and $d\sigma$ is the line measure on $\{y\cdot e_i=\pm 1\}$).
\begin{remark}
 $1$-periodicity has been taken for convenience. We can deal with $l$-periodicity ($l>0$) as well. 
\end{remark}
\paragraph{Weak formulation:}
We first introduce the function spaces
\begin{align*}
L^2_{c,\#}(S) =  \ \{\varphi\in L^2_{loc}(\mathbb{R}^N) \  | \ &\varphi \mbox{ is constant when }y\cdot e_i \in \mathbb{R} \smallsetminus (-1,1)\\
                                                        \mbox{ and }&\varphi \mbox{ is $1$-periodic in }e_j (j\neq i) \mbox{ direction} \},\\
H^1_{c,\#}(S)=\ \{\varphi\in H^1_{loc}(\mathbb{R}^N) \  | \ &\varphi \mbox{ is constant when  }y\cdot e_i \in \mathbb{R} \smallsetminus (-1,1)\\
                                                        \mbox{ and }&\varphi \mbox{ is $1$-periodic in }e_j (j\neq i) \mbox{ direction} \}.
\end{align*}
Here, “$c$” is a floating constant depending on the element under consideration. $L^2_{c,\#}(S)$ and $H^1_{c,\#}(Y)$ are proper 
subspace of $L^2(S)$ and  $H^1(S)$ respectively, and they inherit the subspace norm-topology of the parent space. \\
\\
The motivation for the state space $H^1_{c,\#}$ starting from $a_S(y\cdot e_i)$ is equivalent to $m$ in $e_i$ direction,
but independent of other $y_j$ variables ($j\neq i$). So, in particular, in other directions it is periodic with any period, 
for convenience we take $1$-periodicity. \\
\\
Similarly, one can define $L^2_{c,\#}(\eta;S)$ or $H^1_{c,\#}(\eta;S)$ spaces as follows:
\begin{align*}
L^2_{c,\#}(\eta;S) = \ \{\varphi \in L^2_{loc}(\mathbb{R}^N) \  | \ &e^{-iy\cdot \eta}\varphi \mbox{ is constant when }y\cdot e_i \in \mathbb{R} \smallsetminus (-1,1)\\
                                                        \mbox{ and }&e^{-iy\cdot \eta}\varphi \mbox{ is $1$-periodic in }e_j (j\neq i) \mbox{ direction}, \}\\
H^1_{c,\#}(\eta;S) =\ \{\varphi\in H^1_{loc}(\mathbb{R}^N) \  | \ &e^{-iy\cdot \eta}\varphi \mbox{ is constant when }y\cdot e_i \in \mathbb{R} \smallsetminus (-1,1)\\
                                                        \mbox{ and }&e^{-iy\cdot \eta}\varphi \mbox{ is $1$-periodic in }e_j (j\neq i) \mbox{ direction} \}.
\end{align*}
As a next step we give the weak formulation of the problem in these function spaces. 
We are interested into proving the existence of the eigenvalue and the corresponding eigenvector 
$(\mu(\eta), \varphi_S(y;\eta))$ with $\mu(\eta)\in\mathbb{C}$ and $\varphi_S(\cdot;\eta) \in H^1_{c,\#}(S)$ of 
the following weak formulation of  \eqref{nl7}:
\begin{equation}\label{nl8}
a_S(\eta)(\varphi_S(y;\eta), \psi) = \mu(\eta)(\varphi_S(y;\eta),\psi) \quad \forall \psi\in H^1_{c,\#}(S).
\end{equation}
\paragraph{Existence Result:}
By following the same analysis presented in Section \ref{nl9}, we state the corresponding existence result for this problem \eqref{nl8}.
\begin{proposition}
Fix $\eta \in \mathbb{R}^N$. Then, there exist a sequence of eigenvalues $\{ \mu_m(\eta); m \in \mathbb{N} \}$
and corresponding eigenvectors $\{ \varphi_{S,m}(y;\eta)\in H^1_{c,\#}(S), m \in \mathbb{N} \}$ such that
$$ a_S(\eta)(\varphi_{S,m}(y;\eta), \psi) = \mu_m(\eta)(\varphi_{S,m}(y;\eta),\psi) \quad \forall \psi\in H^1_{c,\#}(S)\ \mbox{ and }\ \forall m \in \mathbb{N}.$$
\end{proposition}
\paragraph{Regularity of the Ground state:}
In the next proposition, we announce the regularity result of the Ground state based on the previous Kato-Rellich analysis.
\begin{proposition}
\begin{enumerate}
 \item Zero is the first eigenvalue of \eqref{nl8} at $\eta=0$ and is an isolated point of the spectrum
with its algebraic multiplicity one.
\item There exists an open neighborhood $\omega^{\prime}$ around zero such that the 
first eigenvalue $\mu_1(\eta)$ is an analytic function on $\omega^{\prime}$ and
there is a choice of the first eigenvector $\varphi_{S,1}(y;\eta)$ satisfying
$$\eta \mapsto \varphi_{S,1}(\cdot;\eta) \in H^1_{c,\#}(S)\mbox{ is analytic on }\omega^{\prime}\,\mbox{ and }\ \varphi_{S,1}(y;0)= |S|^{-1/2},$$
with the boundary normalization condition 
$D^{l}_{\eta}\varphi_{S,1}(y;0)=\ 0$ on the boundary $y\cdot e_i = \pm 1$, for $l\in \mathbb{Z}^N_{+}\smallsetminus \{0\}.$ 
\end{enumerate}
\end{proposition}
\paragraph{Derivatives of $\mu_1(\eta)$ and $\varphi_{S,1}(\eta)$ at $\eta=0$:}
The procedure consists of differentiating the eigenvalue equation \eqref{nl7} 
for $\mu(\eta)=\mu_1(\eta)$ and $\varphi_S(\cdot;\eta)=\varphi_{S,1}(\cdot;\eta)$.
\paragraph{Step 1.\hspace{2pt} Zeroth order derivatives:}
We simply recall that $\varphi_{S,1}(y;0) = |S|^{-1/2}$ by our choice and $\mu_1(0) = 0.$
\paragraph{Step 2.\hspace{2pt}First order derivatives of $\mu_1(\eta)$ at $\eta =0$:}
By differentiating the equation \eqref{nl7} once with respect to $\eta_k$ and then 
taking scalar product with $\varphi_{S,1}(\cdot;\eta)$ in $L^2(S)$ at $\eta=0$ we get
\begin{equation*}\langle D_k (\mathcal{A}_S(0) - \mu_1(0))\varphi_{S,1}(\cdot;0),\varphi_{S,1}(\cdot;0)\rangle = 0.\end{equation*}
Then using $D_k \mathcal{A}_S(0)\varphi_{S,1}(\cdot;0) = iC^S_k\varphi_1(\cdot;0)$ whose integral over $S$ vanishes 
through integration by parts together with using the boundary conditions in \eqref{nl7},
it follows therefore that 
\begin{equation}D_k \mu_1(0) = 0 \quad \forall\hspace{2pt} k = 1,..,N. \end{equation}
\paragraph{Step 3. \hspace{2pt} First order derivatives of $\varphi_{S,1}(.;\eta)$ at $\eta=0$:}
By differentiating \eqref{nl7} once with respect to $\eta_k$ at zero, one has
\begin{align}
-\frac{\partial}{\partial y_l}\Big(a_S(y\cdot e_i)\frac{\partial}{\partial y_l}D_k\varphi_{S,1}(\cdot;0)\Big) &= -iC^S_k\varphi_1(\cdot;0) = i\varphi_{S,1}(\cdot;0)\frac{\partial}{\partial y_k}(a_S(y\cdot e_i))\quad\mbox{in }S, \label{nl11}\\
D_k\varphi_{S,1}(\cdot;0) &=\ 0  \quad\mbox{ on } y\cdot e_i = \pm 1,\label{nl12}\\
D_k\varphi_{S,1}(\cdot;0)\ &\mbox{ is }\ 1\mbox{-periodic in }e_j\ ( j\neq i) \mbox{ direction, }\label{nl13}\\ 
\mbox{and}\quad \int_{y\cdot e_i = \pm 1} a_S(y\cdot e_i)&\lb \nabla_x  D_k\varphi_{S,1}(\cdot;0) + i\varphi_{S,1}(\cdot;0)e_k\rb\cdot\nu \ d\sigma =\ 0 .\label{nl14}
\end{align}
For $k=i$ we seek $D_i\varphi_{S,1}(y;0)=\frac{1}{i}\frac{d}{dy_i}\varphi_{S,1}(y_i;0)$ solving 
the equation \eqref{nl11} uniquely with the Dirichlet boundary condition \eqref{nl12}. 
In that case, \eqref{nl11} becomes an ordinary differential equation and gets identified with \eqref{nl1} to give
\begin{equation*}
\frac{1}{i}\frac{d}{dy_i}\varphi_{S,1}(y_i;0)=\ i|S|^{-1/2}(w_{e_i}(y_i)-y_i)
\end{equation*}
satisfying \eqref{nl14} also.\\
And for $k\neq i $ we notice that $D_k \varphi_{S,1}(y;0)= 0$ is the unique solution of the above system of equations.
\paragraph{Step 4.\hspace{2pt} Second derivatives of $\mu_1(\eta)$ at $\eta=0$:}
We differentiate \eqref{nl7} with respect to $\eta_k$ twice and then, by taking 
scalar product with $\varphi_{S,1}(.;\eta)$ in $L^2(S)$ at $\eta=0$, one obtains 
$$\langle D^2_{kk}(\mathcal{A}_S(0) - \mu_1(0))\varphi_{S,1}(\cdot;0), \varphi_{S,1}(\cdot;0)\rangle 
+ 2\langle [D_k(\mathcal{A}_S(0) - \mu_1(0))]D_k\varphi_{S,1}(\cdot;0), \varphi_{S,1}(\cdot;0)\rangle =\ 0.$$
Which simply becomes 
\begin{equation*}\begin{aligned}
\frac{1}{2}D^2_{ii}\mu_1(0) &= \frac{1}{|S|}\int_{S} a_S(y\cdot e_i)dx - \frac{1}{|S|}\int_{S}C^S_i(w_{e_i}(y_i) - y_i)dy =\ m =\ \underline{a_S},\\
\mbox{and }\quad\frac{1}{2}D^2_{jj}\mu_1(0) &= \frac{1}{|S|}\int_{S} a_S(y\cdot e_i)dx =\ \overline{a_S}\quad\forall j \neq i,
\end{aligned}\end{equation*}
which are indeed the homogenized coefficients governed with the simple laminates in two-phase
medium stated in \eqref{nl5}.
\hfill\qed
\\

As a next step one can define the first Bloch transformation likewise 
in Section \ref{nl10} and successively the limit analysis can be done as it is done in 
Section \ref{nl16}. 

\section{Appendix}
\begin{lemma}\label{A1}
Let us take the ball $B=B(0,1)$ and an open set in $\Omega\in\mathbb{R}^N$, and 
consider a Vitali covering of $\Omega$ with a countable infinite union 
of disjoint balls with center $y^p$ and radius $\epsilon_p$, where
$p\in \mathbb{N}$, i.e. $\Omega \approx \underset{p\in \mathbb{N}}{\cup} (\epsilon_pB + y^{p})$.
Let us consider $f \in H^1(B)$ and define $F = F(x)$ by $F(x) = \epsilon_p f(\frac{x-y^p}{\epsilon_p})$ in $\epsilon_p B + y^p $ 
a.e. on $\Omega$. Then, $F\in H^1(\Omega)$ if and only if the trace of $f$ vanishes over the boundary
$\partial B$, or $f\in H^1_0(B)$.
\end{lemma}
\begin{remark}
If $\ \Omega$ is a finite union of disjoint balls, then by the above 
definition for a given $f\in H^1(B)$, $F$ is always in $H^1(\Omega).$ 
The presence of countable infinite balls is leading to the zero trace 
condition on $f$.  
\end{remark}

\begin{proof}[Proof of Lemma \ref{A1}]
We first notice that $F\in L^2(\Omega)$.
As $F(x) = \epsilon_p f(\frac{x-y^p}{\epsilon_p})$ in $\epsilon_p B + y^p $ a.e. on $\Omega$, then
\begin{align*}
 \int_{\Omega} |F(x)|^2 dx =&\ \sum_p\int_{\epsilon_pB + y^p} \epsilon_p^2 \Big|f\Big(\frac{x-y^p}{\epsilon_p}\Big)\Big|^2 dx\\
                           =&\ \sum_p \epsilon_p^{N+2}\int_{B} |f(y)|^2 dy\\
                           \leq &\ \mathcal{O}(1)||f||^2_{L^2(B)}.
\end{align*}

\begin{figure}
 \begin{center}
  \includegraphics[width = 8cm]{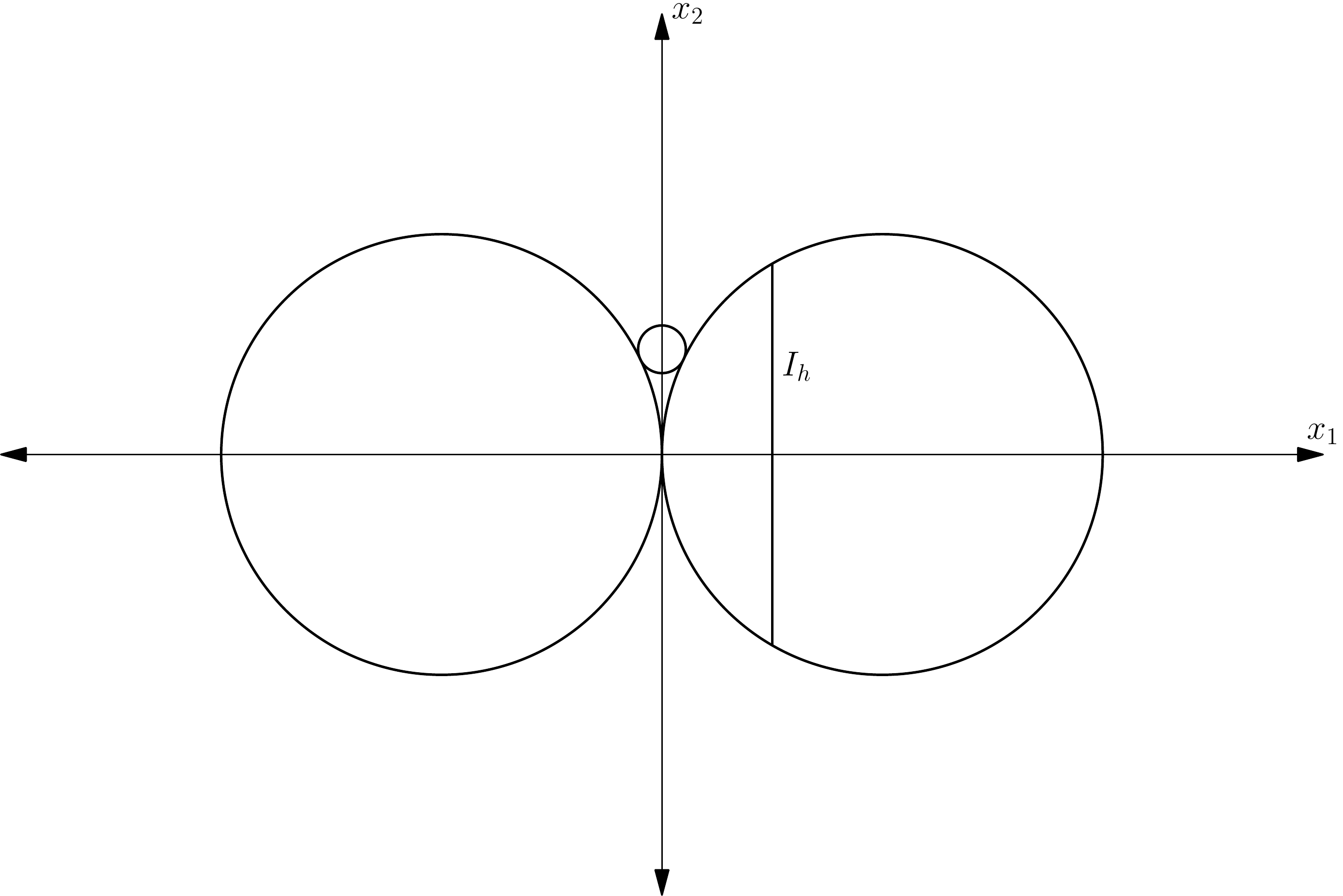}
\end{center}
\caption{}
\end{figure}
\paragraph{Case 1): (only if part)} Let us assume that $f\in H^1_0(B)$.\\
\\
Then the distribution derivative of $F$, say $\nabla F$ can be defined as follows to be in $L^2(\Omega)$:
\begin{equation*}
\nabla_x F(x) = \nabla_y f\Big(\frac{x- y^p}{\epsilon_p}\Big), \quad x\in \epsilon_p B + y^p\ \mbox{a.e. on }\Omega\ \ \mbox{ and }y= \frac{x-y^p}{\epsilon_p}. 
\end{equation*}
It is easy to see that for any $\varphi\in \mathcal{D}(\Omega)$,
\begin{align*}
 -\int_{\Omega} \nabla F \varphi =\ \int_{\Omega} F\nabla\varphi =&\  \sum_p\int_{\epsilon_pB + y^p} \epsilon_p f\Big(\frac{x-y^p}{\epsilon_p}\Big)\nabla\varphi\ dx\\
                              =&\ -\sum_p \int_{\epsilon_pB + y^p} \nabla_y f\Big(\frac{x-y^p}{\epsilon_p}\Big)\varphi\ dx\quad\mbox{ (because }f|_{\partial B} =0)\\
                              =&\ -\int_{\Omega} \sum_p \chi_{(\epsilon_pB+y^p)}\nabla_y f\Big(\frac{x-y^p}{\epsilon_p}\Big)\varphi\ dx,
                              \end{align*}
or
$$\nabla_xF(x) =\ \nabla_y f\Big(\frac{x-y^p}{\epsilon_p}\Big), \quad x\in \epsilon_p B + y^p\ \ \mbox{a.e. on }\Omega $$
and 
\begin{align*}
\int_{\Omega} |\nabla F|^2dx =&\ \sum_p\int_{\epsilon_pB + y^p}  \Big|\nabla_y f\Big(\frac{x-y^p}{\epsilon_p}\Big)\Big|^2 dx = \sum_p \epsilon_p^{N}\int_{B} |\nabla_y f(y)|^2 dy\\
                           =&\ \frac{|\Omega|}{|B|}||\nabla_yf||^2_{L^2(B)}.
\end{align*}
Thus, $F\in H^1(\Omega)$.
\paragraph{Case 2): (if part)} 
Let us assume that $F\in H^1(\Omega)$.\\
\\
\textbf{Step 1.}
For simplicity, we consider $B(0,1)\subset \Omega \subset \mathbb{R}^2$. 
And we assume that $F|_{B} = f \in C^1(\overline{B}).$ Then, we have the following 
claim to establish. \\
\\
\textbf{Claim:}\\
\textit{Origin $(0,0)$ is a Lebesgue point of the trace of $F$ 
on $x_2$-axis, i.e. $\underset{h\rightarrow 0}{\mbox{lim}}
\frac{1}{|I_h|}\int_{I_h} F(0,x_2)dx_2$
exists and it coincides with the value $f(0,0)$, 
where $I_h = (-\sqrt{2h-h^2}, \sqrt{2h-h^2})$}.\\
\\
\textit{Proof of the claim}.
Let us start with recalling the following trace result, which says
\begin{align*}
 F\in H^1(\Omega) \Rightarrow F\in C\big(\overline{\mathbb{R}}_{x_1}^{+}; H^{1/2}({\mathbb{R}}_{x_2})\big).
\end{align*}
So, for $x_2\in I_h= (-\sqrt{2h-h^2}, \sqrt{2h-h^2})$, we have
\begin{align*}
 F(0,x_2) =&\ F(h,x_2) + {o}(1) \quad\mbox{as }h \rightarrow 0\ \ \mbox{ (the $o(1)$ term is small in }H^{1/2}(\mathbb{R}_{x_2}))\\
          =&\ f(h,x_2) + {o}(1).  
\end{align*}
Thus,
$$\frac{1}{|I_h|}\int_{I_h} F(0,x_2)dx_2 =\ \frac{1}{|I_h|}\int_{I_h} f(h,x_2)dx_2 + o(1).$$
Now, by using the uniform continuity of $f$ on $B$, i.e. using $f(h,x_2)-f(0,0) =o(1)$, 
we get
$$ \frac{1}{|I_h|}\int_{I_h} F(0,x_2)dx_2 =\ f(0,0) + o(1).$$
Hence,
$$\quad \underset{h\rightarrow 0}{\mbox{lim}}\ \frac{1}{|I_h|}\int_{I_h} F(0,x_2)dx_2 =\ f(0,0),$$
which establishes our claim.
\paragraph{Step 2.}
Let us consider a Vitali covering of $\Omega$ with 
countable infinite union of disjoint balls $B_p= (\epsilon_pB + y^p)$,
i.e. $\Omega \approx \underset{p\in \mathbb{N}}{\cup}(\epsilon_pB + y^p)$. 
Let us assume that two balls $B$ and some $B_p$ 
are touching each other at origin along the $x_2$ axis.
Then, from the definition of $F$ on $B_p$'s  
and by the above claim, we have
$$ F(0,0)=\ f(0,0) =\ \epsilon_pf\big(\frac{x-y^p}{\epsilon_p}\big)|_{x=0}.$$
It implies that $f(0,0) = 0$. 
Which follows from the fact that if $f(0,0)\neq 0$, then by choosing 
$\epsilon_p$ arbitrarily small, i.e. $\epsilon_p\rightarrow 0$, we 
end up with a contradiction by violating the above second equality. 
\paragraph{Step 3.}
\noindent 
In particular by repeating the same above argument 
for any other points on the boundary $\partial B$, we can show for
$f\in C^1(\overline{B})$, $f|_{\partial B} = 0 $ point wise. 
And the conclusion holds for any dimension also. 
Finally, as $C^1(\overline{B})\cap H^1(B)$ is dense in $H^1(B)$, therefore
for $f\in H^1(B)$, the trace of $f$ vanishes over the boundary $\partial B$
or $f\in H^1_0(B)$. This completes the proof. 
\hfill
\end{proof}

\begin{lemma}\label{A2}
Let us take two open sets $\omega,\Omega\subset\mathbb{R}^N$ and 
consider a Vitali covering of $\Omega$, i.e.
$\Omega \approx \underset{p\in K}{\cup} (\epsilon_{p,n}\omega + y^{p,n})$ with 
$\kappa_n = \underset{p\in K}{sup}\hspace{2pt} \epsilon_{p,n}\rightarrow 0$ 
for a finite or countable $K$ and, for each $n$, the sets $\epsilon_{p,n}\omega + y^{p,n},\ p\in K$ are disjoint.
Let us take $f\in L^2(\omega)$ and define one sequence $f_{n}\in L^2(\Omega)$ as follows
\begin{equation*}
f_{n}(x) = f\Big(\frac{x-y^{p,n}}{\epsilon_{p,n}}\Big)\quad\mbox{in }\epsilon_{p,n}\omega + y^{p,n}.  
\end{equation*}
Then, it converges weakly in $L^2(\Omega)$ to $M_\omega(f)=\ \frac{1}{|\omega|}\int_{\omega} f(y) dy$ (the average of `$f$' over $\omega$). 
\end{lemma}
\begin{proof}
By rescaling, it simply follows that
\begin{align*}
\int_{\Omega} |f_{n}|^2 dx =& \sum_{p\in K} \int_{\epsilon_{p,n}\omega + y^{p,n}} \Big|f\Big(\frac{x-y^{p,n}}{\epsilon_{p,n}}\Big)\Big|^2 dx \\
                                    =& \Big(\int_{\omega} |f|^2 dx\Big) \sum_{p\in K}\epsilon_{p,n}^N  \quad\mbox{and }\sum_{p\in K} \epsilon_{p,n}^N = \frac{|\Omega|}{|\omega|}, 
\end{align*}
which shows that $f_{n}$ is a bounded sequence in $L^2(\Omega)$. \\
\\
We denote $M_\omega(f)$  the average of $f$ on $\omega$, i.e.  $M_\omega(f)= \frac{1}{|\omega|}\int_{\omega} f(y) dy$.\\
Let us consider any smooth function $\varphi\in C_c(\Omega),$ then we want to show
\begin{equation}\label{cls}
 \int_\Omega f_{n}(x)\varphi(x) dx \rightarrow M_\omega(f)\int_\Omega \varphi(x) dx,
\end{equation}
which is enough in order to have the desired weak convergence in $L^2(\Omega)$ by following density arguments.\\
\\
We have
\begin{equation*}
\Big| \int_{\epsilon_{p,n}\omega + y^{p,n}} f_{n}(x)(\varphi(x) - \varphi(y^{p,n})) dx \Big|
\leq\ \epsilon^N_{p,n} M_\omega(|f|) \underset{x, x^{\prime} \in \epsilon_{p,n}\omega + y^{p,n}}{\ max\ } |\varphi(x) - \varphi(x^{\prime})|,
\end{equation*}
then, by taking sum over $p$,  we get 
\begin{equation}\label{lra}
\Big| \int_{\Omega} f_{n}(x)\varphi(x) - M_\omega(f)\sum_p\epsilon_{p,n}^N \varphi(y^{p,n}) dx \Big|
\leq\ \frac{|\Omega|}{|\omega|} M_\omega(|f|) \underset{|x-x^{\prime}|\leq  \kappa_n} {\ max\ } |\varphi(x) - \varphi(x^{\prime})|. 
\end{equation}
Since $\varphi$ is uniformly continuous on $\Omega$, then in the right hand side of \eqref{lra}, we have
\begin{equation*} \underset{|x-x^{\prime}|\leq  \kappa_n} {\ max\ } |\varphi(x) - \varphi(x^{\prime})| =o(1)\end{equation*}
and in the left hand side of \eqref{lra} the Riemann sum satisfies 
\begin{equation*} \sum_p \epsilon_{p,n}^N \varphi(y^{p,n}) = \int_\Omega \varphi(x) dx + o(1).\end{equation*}
Thus, we get  our desired result. 
\hfill\end{proof}
\begin{remark}
The above lemma can be easily extended to any $L^p$ spaces, for any $1\leq p\leq \infty$ (for $p=\infty$ we consider weak* convergence).   
\end{remark}

\subsection*{Acknowledgments}
C. Conca and J. San Mart\'in were partially supported by BASAL-CMM
1030 Project.

\bibliographystyle{plain}
\bibliography{Master_bibfile}
\end{document}